%% file: ishiwata2.tex
\documentclass[11pt,a4paper]{article}
\usepackage{etex}
\usepackage{eurosym}
\usepackage{amsthm}
\usepackage{amsfonts}
\usepackage{amssymb,amsmath}
\usepackage{graphicx}
\usepackage{pictex}
\usepackage{multirow}
\usepackage{hyperref}

\newtheorem{theorem}{Theorem}[section]
\newtheorem{lemma}[theorem]{Lemma}
\newtheorem{corollary}[theorem]{Corollary}
\newtheorem{proposition}[theorem]{Proposition}
\newtheorem{conjecture}[theorem]{Conjecture}
\theoremstyle{definition}
\newtheorem{remark}[theorem]{Remark}
\newtheorem{definition}[theorem]{Definition}
\newtheorem{example}[theorem]{Example}
\renewenvironment{proof}[1][Proof]{\textbf{#1.} }{\ \rule{0.5em}{0.5em}}

\renewcommand{\theequation}{\thesection.\arabic{equation}}

\newenvironment{notation}{\smallskip{\sc Notation.}}{\smallskip}
\input{macro}

\input{tcilatex}

\begin{document}

\title{Poincar\'e constant on manifolds with ends}
\author{Alexander Grigor'yan\thanks{%
Funded by the Deutsche Forschungsgemeinschaft (DFG, German Research Foundation) - SFB 1283/2 2021 - 317210226} \\
Department of Mathematics \\
University of Bielefeld\\
33501 Bielefeld, Germany \\
grigor@math.uni-bielefeld.de \and Satoshi Ishiwata\thanks{%
Partially supported by JSPS KAKENHI 17K05215 and 22K03280} \\
Department of Mathematical Sciences\\
Yamagata University\\
Yamagata 990-8560, Japan\\
ishiwata@sci.kj.yamagata-u.ac.jp \and Laurent Saloff-Coste \thanks{%
Partially supported by NSF grant DMS-1707589 and DMS-2054593} \\
Department of Mathematics\\
Cornell University\\
Ithaca, NY, 14853-4201, USA\\
lsc@math.cornell.edu}
\date{\today }
\maketitle

\begin{abstract}
We obtain optimal estimates of the Poincar\'e constant of central balls on
manifolds with finitely many ends. Surprisingly enough, the Poincar\'e constant is
determined by the \emph{second} largest end. The proof is based on the argument 
by Kusuoka-Stroock where the heat kernel estimates on the central balls play an
essential role. For this purpose, we extend earlier heat kernel estimates obtained by the 
authors 
 to a larger class of parabolic manifolds with ends.
\end{abstract}

\tableofcontents


\section{Introduction}

\label{introduction}

Let $M$ be a Riemannian manifold. Denote by $\mu$ the Riemannian
measure on $M$ and by $\nabla$ the gradient.
For a precompact connected open set $U\subset M$, define the \textit{%
Poincar\'{e} constant} $\Lambda (U)$ 
as the smallest number such that the following inequality holds for all $%
f\in C^{1}(\overline{U})$: 
\[
\int_{U}|f-f_{U}|^{2}d\mu \leq \Lambda (U)\int_{U}|\nabla f|^{2}d\mu ,
\]%
where $f_{U}:=\frac{1}{\mu (U)}\int_{U}fd\mu $. Equivalently, we have 
\[
\Lambda (U)= \frac{1}{\lambda (U)},
\]%
where $\lambda(U)$ is the smallest positive eigenvalue of $-\Delta $ 
in $U$ with the Neumann condition on $\partial U$. Here $\Delta$ is the 
Laplace-Beltrami operator on $M$. 
Estimating the Poincar\'e constant has many applications. See \cite{Bate},
 \cite{Besson}, \cite{Carron 2022}, \cite{Colding}, \cite{Coulhon 2020},  \cite{SC LNS} for examples and references therein.

Denote by $d\left( x,y\right) $ the geodesic distance on $M$ and by $B\left(
x,r\right) $ -- open geodesic balls on $M$. In this paper we are concerned
with estimating the Poincar\'{e} constant $\Lambda \left( B\left(
x,r\right) \right) $. It is well-known that in $\mathbb{R}^{n}$ 
\[
\Lambda (B(x,r))=C_{n}r^{2}.
\]%
It is also known by \cite{Li-Yau} that on complete non-compact manifolds 
with non-negative Ricci curvature
\begin{equation}
\Lambda(B(x,r)) \simeq r^2.
\label{LYPI}
\end{equation}

There are other classes of manifolds satisfying (\ref{LYPI}),
for example, Lie groups of polynomial volume growth 
(see \cite{Coulhon-SC} and \cite{Grigoryan 1991}), 
the tube manifold around the square lattice $\mathbb{Z}^d$ (or jungle gym).

However, there are natural examples of manifolds where (\ref{LYPI}) does not hold,
for example, the hyperbolic spaces where $\Lambda (B(x,r))$
grows exponentially in $r$. Another example that is more relevant for this
paper is the connected sum $\mathbb{R}^{n}\#\mathbb{R}^{n}$, where $n\geq 2$%
. By $\mathbb{R}^{n}\#\mathbb{R}^{n}$ we denote any manifold that is
obtained by gluing together two copies of $\mathbb{R}^{n}$ over a compact
tube (see Fig. \ref{figure: Rn+Rn}). 
\begin{figure}[tbph]
\begin{center}
\scalebox{0.8}{
\includegraphics{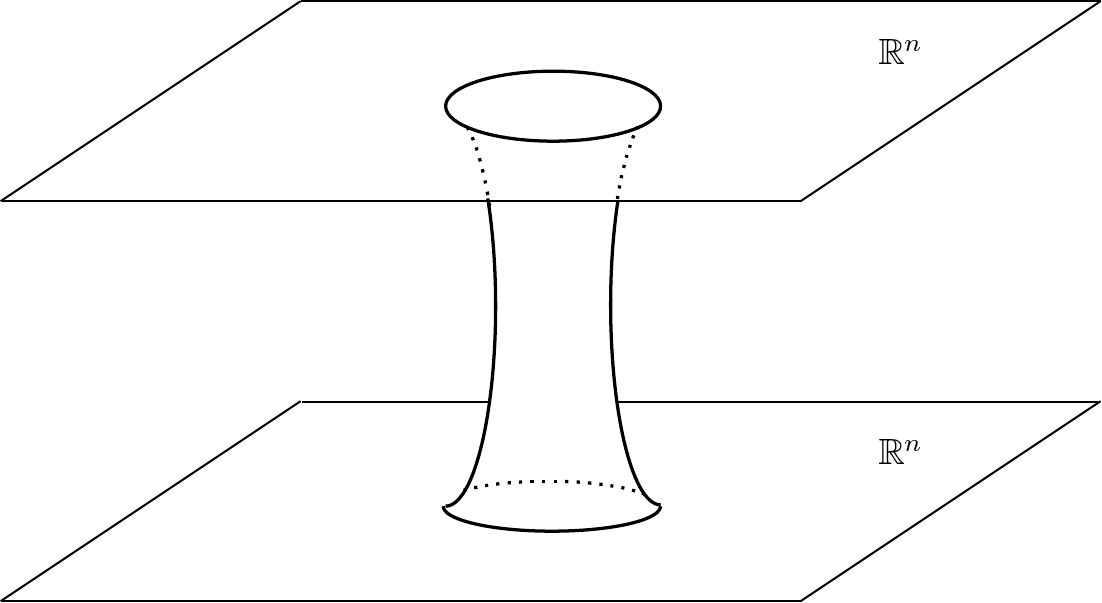} 
}
\end{center}
\caption{$\mathbb{R}^n\# \mathbb{R}^n$}
\label{figure: Rn+Rn}
\end{figure}

It follows from the results of this paper (Theorems \ref{main non-parab}  and \ref{main parab}) 
that on such a manifold%
\[
\Lambda (B(o,r))\simeq \left\{ 
\begin{array}{ll}
r^{n} & \mbox{if }n>2, \\ 
r^{2}\log r & \mbox{if }n=2,%
\end{array}%
\right.
\]%
that is,
\[
\lambda (B(o,r))\simeq \left\{ 
\begin{array}{ll}
\frac{1}{r^{n}} & \mbox{if }n>2, \\ 
\frac{1}{r^{2}\log r} & \mbox{if }n=2%
\end{array}%
\right.
\]%
for all large $r$ and for a central reference point $o\in \mathbb{R}^{n}\#%
\mathbb{R}^{n}$ (see Section \ref{SecExamples}). 
In these simple cases, these results are well-known to the specialists of the subject.

Consider now a connected sum $M=M_1 \# \cdots \# M_k$ of $k$ model 
manifolds $M_1, \ldots, M_k$ with a same dimension $N$
(see Section \ref{SecExamples} for details of this construction). 
For example, $M_i$ can be a surface of revolution (see Fig. 
\ref{figure: ex2}).
Assume that the volume growth function $V_i(r)$ of $M_i$ satisfies
 for some $\alpha_i>0$ and $\beta_i \in \mathbb{R}$
\[
V_i(r)\simeq r^{\alpha_i}(\log r)^{\beta_i}, ~~(i=1, \ldots ,k,  ~r\gg 1).
\] 
We assume that 
\[
(N , 0) \succeq (\alpha_1, \beta_1) \succeq 
(\alpha_2, \beta_2) \succeq \cdots \succeq (\alpha_k, \beta_k)
\]
in the sense of the lexicographical order which implies that
\[V_1(r) \gtrsim V_2(r) \gtrsim \cdots \gtrsim V_k(r)
\]
(see Fig. \ref{figure: ex2}). 
It follows from the main results of this paper (Theorems \ref{main non-parab}  and \ref{main parab}) 
that the Poincar\'e constant 
$\Lambda (B(o,r))$ on $M$ is determined, quite surprisingly, solely
by the {\bf{\emph{second largest}}} end $M_2$: 
\[
\Lambda (B(o,r))\simeq \left\{ 
\begin{array}{ll}
r^{\alpha _{2}}(\log r)^{\beta _{2}} & \mbox{if }(\alpha _{2},\beta
_{2})\succ (2,1), \\ 
r^{2} \left(\log r \right) \left(\log \log r\right) & \mbox{if }(\alpha _{2},\beta _{2})=(2,1), \\ 
r^{2}\log r & \mbox{if }\alpha _{2}=2,\beta _{2}<1, \\ 
r^{2} & \mbox{if }\alpha _{2}<2%
\end{array}%
\right.
\] 
for all large $r$ (see Example \ref{ex11} for details).

\begin{figure}[tbph]
\begin{center}
\scalebox{0.6}{
\includegraphics{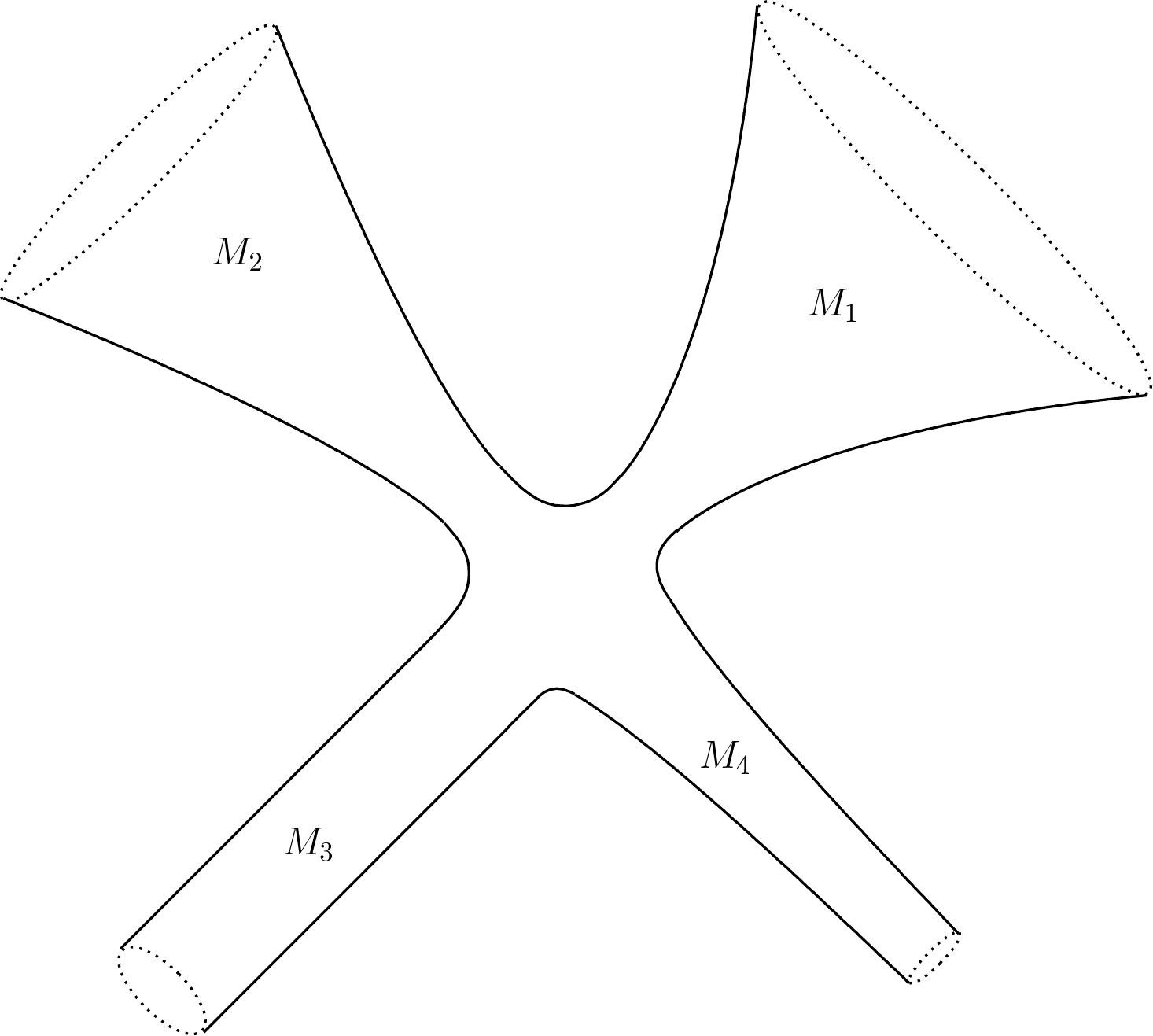} 
}
\end{center}
\caption{Connected sum of model manifolds $M_1, M_2, M_3, M_4$}
\label{figure: ex2}
\end{figure}

Let us mention for comparison that if $V_i(r) \simeq r^{\alpha_i}$ then 
the heat kernel long time behavior is determined by the end $M_i$ having $\alpha_i$ 
nearest to $2$ (see Example \ref{example heat}).

In the next section, we describe a more general class of manifolds with ends where we obtain 
two sided estimates of the Poincar\'e constant. The main results are stated in 
Theorems \ref{main non-parab} and \ref{main parab}. Proofs are given in Sections 
\ref{SecPI} and \ref{section Whitney} 
and use crucially heat kernel estimates.

\begin{notation}
The notation $f\simeq g$ for two non-negative functions $f,g$ means
that there are two positive constants $c_{1},c_{2}$ such that $c_{1}g\leq
f\leq c_{2}g$ for the specified range of the arguments of $f$ and $g$.
Similarly, the notation $f\gtrsim g$ (resp. $f\lesssim g$) means that there
is a positive constant $c$ such that $cf\geq g$ (resp. $f\leq cg$) .
Throughout this article, the letters $c,C,b,...$ denote
positive constants whose values may be different at different instances.
When the value of a constant is significant, it will be explicitly stated. 
\end{notation}



\section{Main results}\setcounter{equation}{0}

\subsection{Heat kernels}
\label{subsection Heat kernels}
Let us recall some known results about the heat kernel
on manifolds. Let $M$ be a Riemannian manifold. Denote by $\mathrm{vol}$ the Riemannian
measure on $M$. Given a smooth positive function $\sigma $ on $M$, define a
measure $\mu $ on $M$ by $d\mu =\sigma d\mathrm{vol}$. The pair $(M,\mu )$
is called a \textit{weighted manifold}. Any Riemannian manifold can be
considered as a weighted manifold with $\sigma =1$.
The Laplace operator $\Delta $ of the weighted manifold $(M,\mu )$ is
defined by 
\[
\Delta =\frac{1}{\sigma }\mathrm{div}(\sigma \nabla ),
\]%
where $\mathrm{div}$ and $\nabla $ are the divergence and the gradient of
the Riemannian metric of $M$. The operator $\Delta $ is known to be
symmetric with respect to the measure $\mu $ (see \cite{G AMS}).

Set $V\left( x,r\right) =\mu \left( B\left( x,r\right) \right) $.

\begin{definition}
We say that a weighted manifold $M$ satisfies the \emph{volume doubling} condition %
{\rm{(VD)}} if there exists a constant $C$ such that, for all $x\in M$ and $r>0$,%
\[
V\left( x,2r\right) \leq CV\left( x,r\right) .
\]
\end{definition}

\begin{definition}
We say that a weighted manifold $M$ admits the \emph{scale invariant Poincar\'{e}
inequality} {\rm{(PI)}} if there exist constants $C>0$ and $\kappa \geq 1$ such
that, for all $x\in M$ and $r>0,$ 
and for all $f\in C^{1}(\overline{B(x,\kappa r)})$, 
\[
\int_{B(x,r)}|f-f_{B(x,r)}|^{2}d\mu \leq Cr^{2}\int_{B(x,\kappa r)}|\nabla
f|^{2}d\mu .
\]
\end{definition}

Denote by $p\left( t,x,y\right) $ the heat kernel of $M$, that is, the
minimal positive fundamental solution of the heat equation $\partial
_{t}u=\Delta u$. The following theorem is a combined result of \cite%
{Grigoryan 1991}, \cite{SC 1992} based on previous contributions of Moser 
\cite{Moser}, Kusuoka--Stroock \cite{KS} et al.

\begin{theorem}
\label{TVD+PI}On a geodesically complete, non-compact weighted manifold $M$,
the following conditions are equivalent:

\begin{enumerate}
\item[(i)] {\rm{(PI)}} and {\rm{(VD)}}.

\item[(ii)] The Li-Yau type heat kernel estimates: 
\begin{equation}
p(t,x,y)\asymp \frac{C}{V(x,\sqrt{t})}\exp \left( -b\frac{d^{2}(x,y)}{t}%
\right) ,  \label{LY}
\end{equation}%
where the sign $\asymp $ means that both $\leq $ and $\geq $ hold but with
different values of the positive constants $C$ and $b$.

\item[(iii)] The uniform parabolic Harnack inequality (for the definition
see \cite[Section 2.1]{G-SC stability}).
\end{enumerate}
\end{theorem}

The estimate (\ref{LY}) was proved for the first time by Li and Yau \cite%
{Li-Yau} on manifolds of non-negative Ricci curvature.

\begin{definition}\label{nice}
We say that a manifold $M$ is \textit{nice} if  $M$ is a geodesically complete, non-compact weighted manifold 
which admits the Poincar\'{e}
inequality (PI) and the volume doubling condition (VD), so that $M$
satisfies each of the conditions of Theorem \ref{TVD+PI}.
\end{definition}

\subsection{Manifold with ends}
\label{subsection ends}
Fix a natural number $k\geq 2$. Let $M_{1},...,M_{k}$ be a sequence of
geodesically complete, non-compact weighted manifolds of the same dimension.

\begin{definition}
Let $M$ be a weighted manifold. We say that $M$ is a manifold with $k$ ends $%
M_{1},\ldots , M_{k}$ and write 
\begin{equation}
M=M_{1}\#...\#M_{k}  \label{M1Mk}
\end{equation}%
if there is a compact set $K\subset M$ so that $M\setminus K$ consists of $k$
connected components $E_{1},E_{2},\ldots ,E_{k}$ such that each $E_{i}$ is
isometric (as a weighted manifold) to $M_{i}\setminus K_{i}$ for some
compact set $K_{i}\subset M_{i}$ (see Fig. \ref{figure: connectedsum}). Each 
$E_{i}$ will be referred to as an \emph{end} of $M$.
\end{definition}

\begin{figure}[tbph]
\begin{center}
\scalebox{0.8}{
\includegraphics{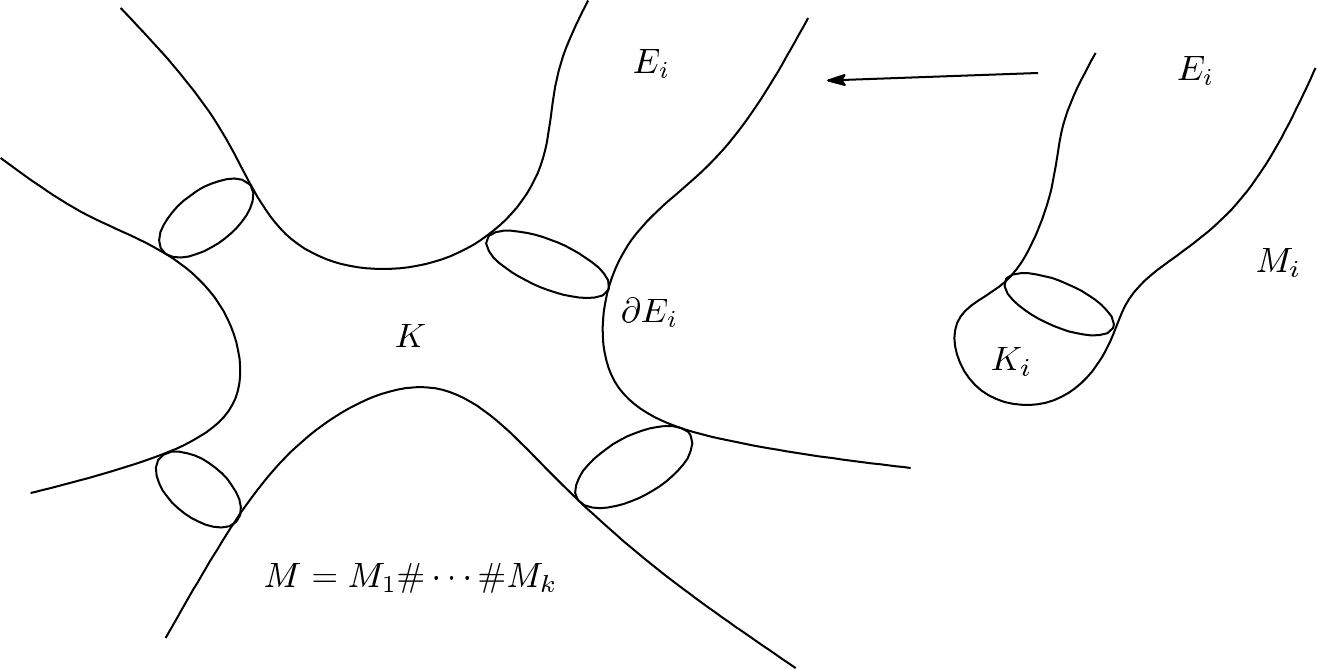} 
}
\end{center}
\caption{Manifold with ends}
\label{figure: connectedsum}
\end{figure}

For the manifold (\ref{M1Mk}) we will also use the notation $M=\#_{i\in
I}M_{i}$ where 
\[
I=\left\{ 1,2,...,k\right\} .
\]

It will be important for us to distinguish between \emph{parabolic} and 
\emph{non-parabolic} manifolds.

\begin{definition}
We say that a weighted manifold $N$ is \textit{parabolic} if the weighted
Laplace operator $\Delta $ has no positive Green function$.$ Equivalently $N$
is parabolic if and only if%
\begin{equation}
\int^{\infty }p\left( t,x,y\right) dt=\infty  \label{intpt}
\end{equation}%
for all/some $x,y\in N$. 
\end{definition}

If the integral in (\ref{intpt}) converges then it determines the minimal
positive Green function%
\[
g\left( x,y\right) =\int_{0}^{\infty }p(t,x,y)dt.
\]%
If $N$ satisfies (VD) and (PI), then $N$ is parabolic if and only if 
\[
\int^{\infty }\frac{rdr}{V\left( x,r\right) }=\infty
\]%
for all/some $x\in N$.

For a manifold with ends (\ref{M1Mk}), we say that an end $E_{i}$ is
parabolic (or non-parabolic) if $M_{i}$ is parabolic (resp., non-parabolic).
It is easy to verify that $M$ is parabolic if and only if all the ends $%
E_{i} $ are parabolic.


In the sequel, we always assume that each end $M_{i}$ is nice (see Definition \ref{nice}). 
Besides, if $M_{i}$
is parabolic then we assume in addition that $M_{i}$ satisfies the following
condition (RCA). 


\begin{definition}[RCA]
We say that a Riemannian manifold $N$ has \emph{relatively connected annuli}
(shortly {\rm{(RCA)}}) with respect to a reference point $o\in N$ if there exist a
constant $A>1$ such that, for any $r>A^{2}$, any two points $x,y\in N$
satisfying $d(x,o)=d(y,o)=r$, are connected by a continuous path in $%
B(o,Ar)\setminus B(o,A^{-1}r)$.
\end{definition}

For each $i=1,...,k$, let $\mu _{i}$ be the reference measure on $M_{i}$ and 
$d_{i}$ be the geodesic distance on $M_{i}$. Denote by $B_{i}\left(
x,r\right) $ geodesic balls on $M_{i}$ and set $V_{i}\left( x,r\right) =\mu
_{i}\left( B_{i}\left( x,r\right) \right) $. Fix a reference point $o_{i}\in
K_{i}$. Set 
\begin{equation}
V_{i}\left( r\right) =V_{i}\left( o_{i},r\right)  \label{Vi}
\end{equation}%
and refer to $V_{i}$ as the (pointed) volume function of $M_{i}$. Note that
for small $r$, we have $V_{i}(r)\simeq V_{j}(r)\simeq r^{n}$ where $n=\dim
M. $


Clearly, $M_{i}$ is parabolic if and only if 
\[
\int^{\infty }\frac{rdr}{V_{i}\left( r\right) }=\infty .
\]%
Define a function $h_{i}\left( r\right) $ for all $r>0$ by%
\begin{equation}
h_{i}\left( r\right) =1+\left( \int_{1}^{r}\frac{sds}{V_{i}(s)}\right) _{+}.
\label{hi}
\end{equation}%
Then $M_{i}$ is parabolic if $h_{i}\left( r\right) \rightarrow \infty $ as $%
r\rightarrow \infty $ and non-parabolic if $h_{i}\simeq 1.$

\begin{definition}\label{def subcritical}
We say that $M_i$ is \emph{subcritical} if for all $r\gg 1$
\begin{equation}
h_i(r) \simeq \frac{r^2}{V_i(r)}.\label{hiVi}
\end{equation}
\end{definition}
Note that the subcriticality condition (\ref{hiVi}) is equivalent to 
\[
h_{i}(r)\lesssim \frac{r^{2}}{V_{i}\left( r\right) },
\]%
since the opposite inequality 
\begin{equation}
h_{i}\left( r\right) \gtrsim \frac{r^{2}}{V_{i}\left( r\right) }  \label{hi>}
\end{equation}
follows trivially from (\ref{hi}) and the doubling property of $V_{i}\left(
r\right) .$

For example, if $V_{i}\left( r\right) \simeq r^{\alpha }$ for large $r$ then%
\[
h_{i}(r)\simeq \left\{ 
\begin{array}{ll}
r^{2-\alpha } & \mbox{if }\alpha <2, \\ 
\log r & \mbox{if }\alpha =2, \\ 
1 & \mbox{if }\alpha >2,%
\end{array}%
\right.
\]%
for large $r$. In this case $M_{i}$ is parabolic if and only if $\alpha \leq
2$ and subcritical if and only if $\alpha<2$.

If $V_{i}\left( r\right) \simeq r^{2}\left( \log r\right)^{\beta}$ for large $r$ then 
\[
h_{i}(r)\simeq \left\{ 
\begin{array}{ll}
(\log r)^{1-\beta } & \mbox{if }\beta <1, \\ 
\log \log r & \mbox{if }\beta =1, \\ 
1 & \mbox{if }\beta >1,%
\end{array}%
\right.
\]%
for large $r$. In this case $M_{i}$ is non-subcritical for all $\beta \in \mathbb{R}$ and parabolic if and only if $\beta \leq
1$.

In the next section we state the main results of this paper about the Poincar%
\'{e} constant on manifolds with ends. For all Poincar\'{e} type estimates
obtained in this article, the range of interest is $r\gg 1$. Similarly, for
all heat kernel estimates, the range of interest is $t\gg 1$.

\subsection{Poincar\'e constant}
First let us state a general definition of Poincar\'e constant. 
Let $(M,\mu)$ be a weighted manifold. 
For precompact connected open sets $U\subset U^{\prime }\subset M$, define the \textit{%
Poincar\'{e} constant} $\Lambda (U,U^{\prime })$ of the couple ($U,U^{\prime
}$) as the smallest number such that the following inequality holds for all $%
f\in C^{1}(\overline{U^{\prime }})$: 
\[
\int_{U}|f-f_{U}|^{2}d\mu \leq \Lambda (U,U^{\prime })\int_{U^{\prime
}}|\nabla f|^{2}d\mu ,
\]%
where $f_{U}:=\frac{1}{\mu (U)}\int_{U}fd\mu $. Equivalently, we have 
\begin{equation}
\Lambda (U,U^{\prime })=\sup_{\stackrel{ f\in C^{1}(\overline{U^{\prime }}) }{f\neq \mathrm{const}}}
\frac{\inf_{\xi \in \mathbb{R}}\int_{U}|f-\xi
|^{2}d\mu }{\int_{U^{\prime }}|\nabla f|^{2}d\mu }.  \label{Lam sup}
\end{equation}%
We note that the Poincar\'e inequality (PI) is equivalent to hold for all $x \in M$, and $r>0$,
\[
\Lambda(B(x, r), B(x, \kappa r)) \lesssim r^2.
\]
The function $(U,U^{\prime })\rightarrow \Lambda (U,U^{\prime })$ is clearly
monotone in the following sense: if $W\subset U\subset U^{\prime }\subset
W^{\prime }$ then $\Lambda \left( W,W^{\prime }\right) \leq \Lambda \left(
U,U^{\prime }\right) .$
If $U=U^{\prime }$, then we use a shorter notation%
\[
\Lambda \left( U\right) :=\Lambda \left( U,U\right) .
\]%
In this case the number 
\begin{equation}
\lambda (U):=\frac{1}{\Lambda (U)}=\inf_{\stackrel{ f\in C^{1}(\overline{U})}
{f\neq \mathrm{const}}}\frac{\int_{U}|\nabla f|^{2}d\mu }{%
\int_{U}|f-f_{U}|^{2}d\mu }  \label{ev}
\end{equation}%
is the spectral gap of $-\Delta $ on $U$, that is, the smallest positive
eigenvalue of $-\Delta $ in $U$ with the Neumann condition on $\partial U$.

Let $M=\#_{i\in I}M_{i}$ be a manifold with nice ends as described above.
For any 
$r>0$ define natural numbers $m=m(r)$ and $n=n(r)$ so that 
\begin{equation}
V_{m}(r)=\max_{i\in I}V_{i}(r)
\label{largest end}
\end{equation}%
and 
\begin{equation}
V_{n}(r)=\max_{i\in I\backslash \{m\}}V_{i}(r).
\label{second largest end}
\end{equation}%
That is, $V_{m}(r)$ is the largest volume function at scale $r$, and $V_{n}(r)$ is
the \emph{second largest} volume function at scale $r$.

Fix a central reference point $o \in K$.
We first state our result in the case when all $M_{i}$ are non-parabolic.

\begin{theorem}[All ends are non-parabolic]
\label{main non-parab}Let $M=\#_{i\in I}M_{i}$ be a manifold with nice ends.
Assume that all $M_{i}$ are non-parabolic.
Then, for all $r\gg 1,$ 
\begin{equation}
\Lambda (B(o,r))\lesssim V_{n}(r), \label{PC upper non-parab}
\end{equation}%
where $n$ is defined in (\ref{second largest end}). 
If, in addition each $V_{i}$ satisfies for all $r\gg 1$%
\[
rV_{i}^{\prime }(r)\lesssim V_{i}(r),
\]%
then, for all $r\gg 1$,%
\[
\Lambda (B(o,r))\simeq V_{n}(r).
\]
\end{theorem}

As this theorem shows, the Poincar\'{e} constant is
determined by the \emph{second largest} volume function at scale $r$
via the simple formula $\Lambda(B(o,r)) \simeq V_n(r)$.

The case when some ends are parabolic is more challenging and our results
are less complete. In this case, we will need the following definition which 
requires to subdivide the set of ends into three subsets $I_{super}$, $I_{middle}$ and 
$I_{sub}$. The subset $I_{super}$ is made of ``large ends''. The set of $I_{sub}$ is made of 
``small ends''. Here, large and small are defined by two fixed positive parameter $\epsilon$ 
and $\delta$. Restrictive hypothesis are made on the ends in $I_{middle}$.
\begin{definition}[COE]
\label{DefCOE}
Let $M=\#_{i \in I} M_i$ be a manifold with nice ends. 
We say that  $M$ has \emph{%
critically ordered ends} {\rm{(COE)}} if there exist $\varepsilon ,\delta ,\gamma
_{1},\gamma _{2}>0$ such that 
\begin{equation}
\gamma _{1}<\varepsilon ,~~\gamma _{1}+\gamma _{2}<\delta <2,~~2\gamma
_{1}+\gamma _{2}<2,  \label{gade}
\end{equation}%
and a decomposition 
\begin{equation}
I=I_{super}\sqcup I_{middle}\sqcup I_{sub}  \label{I=}
\end{equation}%
such that the following conditions are satisfied:

\begin{itemize}
\item[$\left( a\right) $] For each $i\in I_{super}$ and all $r\geq 1$, 
\[
V_{i}(r)\gtrsim r^{2+\epsilon }\ .
\]

\item[$\left( b\right) $] For each $i\in I_{sub}$, $V_{i}$ is subcritical (see Definition \ref{def subcritical}) 
and 
\begin{equation}
V_{i}(r)\lesssim r^{2-\delta }\ .  \label{2-de}
\end{equation}

\item[$\left( c\right) $] For each $i\in I_{middle}$, 
\begin{equation}
\left( \frac{R}{r}\right) ^{2-\gamma _{2}}\lesssim \frac{V_{i}(R)}{V_{i}(r)}%
\lesssim \left( \frac{R}{r}\right) ^{2+\gamma _{1}}~\mbox{for all }1\leq
r\leq R.  \label{middle}
\end{equation}%
For any pair $i,j\in I_{middle}$, $V_{i}\gtrsim
V_{j} $ or $V_{j}\gtrsim V_{i}$ (i.e., the ends in $I_{middle}$ can be
ordered according to their volume growth uniformly over $r\in \lbrack
1,\infty )$) and $V_{i}\gtrsim V_{j}$ implies that $V_{i}h_{i}\gtrsim
V_{j}h_{j}$. Finally, if $M$ is parabolic (i.e., all ends are parabolic)
then we assume that 
$V_{i}\gtrsim V_{j}$ also implies $V_{i}h_{i}^{2}\lesssim V_{j}h_{j}^{2}
$.
\end{itemize}
\end{definition}

Note that  (\ref{hiVi}) implies (\ref{2-de}) with \emph{some} $\delta
>0$ (see Section \ref{Secsubcritical}), but here we need $\delta $ also to
satisfy (\ref{gade}).

%
%
%
%
%

\begin{example}
\label{Exaibi}Let all functions $V_{i}\left( r\right) $ have for $r\gg 1$
the form 
\begin{equation}
V_{i}(r)=r^{\alpha _{i}}\left( \log r\right) ^{\beta _{i}}  \label{aibi}
\end{equation}%
for some $\alpha _{i}>0$ and $\beta _{i}\in \mathbb{R}$. Let us show that $M$
satisfies {\rm{(COE)}}. Define $I_{super}$ to consist of all indices 
$i$ such that $\alpha _{i}>2$, $I_{sub}$ to consist of all $i$ such that $%
\alpha _{i}<2$, and $I_{middle}$ to consists of all $i$ with $\alpha _{i}=2$%
. Clearly, $\left( a\right) $ and $\left( b\right) $ are satisfied. Let us
verify $\left( c\right) .$ Indeed, in the case of $\alpha _{i}=2$, we have%
\[
h_{i}(r)\simeq \left\{ 
\begin{array}{ll}
1, & \mbox{if }\beta _{i}>1 \\ 
\log \log r & \mbox{if }\beta _{i}=1 \\ 
\left( \log r\right) ^{1-\beta _{i}} & \mbox{if }\beta _{i}<1%
\end{array}%
\right.
\]%
and, hence,%
\[
V_{i}(r)h_{i}(r)\simeq \left\{ 
\begin{array}{ll}
r^{2}\left( \log r\right) ^{\beta _{i}} & \mbox{if }\beta _{i}>1 \\ 
r^{2}\log r\log \log r & \mbox{if }\beta _{i}=1 \\ 
r^{2}\log r & \mbox{if }\beta _{i}<1.%
\end{array}%
\right.
\]%
We see that 
\[
V_{i}\gtrsim V_{j}\Leftrightarrow \beta _{i}\geq \beta _{j}\Leftrightarrow
V_{i}h_{i}\gtrsim V_{j}h_{j}.
\]%
Moreover, if $M$ is parabolic then all $\beta _{i}\leq 1$ and in this case%
\[
V_{i}(r)h_{i}^{2}(r)\simeq \left\{ 
\begin{array}{ll}
r^{2}\log r\left( \log \log r\right) ^{2} & \mbox{if }\beta _{i}=1 \\ 
r^{2}\left( \log r\right) ^{2-\beta _{i}} & \mbox{if }\beta _{i}<1%
\end{array}%
\right.
\]%
so that%
\[
V_{i}\gtrsim V_{j}\Leftrightarrow \beta _{i}\geq \beta _{j}\Leftrightarrow
V_{i}h_{i}^{2}\lesssim V_{j}h_{j}^{2}.
\]%
Hence, $M\ $satisfies {\rm{(COE)}}. Further examples of such manifolds can be
found in Section \ref{SecExamples}.
\end{example}

Now we can state our main result in the case when the ends may be parabolic.
\begin{theorem}[At least one end is parabolic]
\label{main parab}Let $M=\#_{i\in I}M_{i}$ be a manifold with nice ends.
Assume that each parabolic end $M_{i}$
satisfies also {\rm{(RCA)}} with a reference point $o_{i}$. 
Assume further that $%
M $ satisfies {\rm{(COE)}}. Then, for all $r\gg 1$, we have 
\begin{equation}
\Lambda (B(o,r))\lesssim V_{n}(r)h_{n}(r),  \label{PC upper parab}
\end{equation}%
where $n$ is defined in (\ref{second largest end}). 
If, in addition each $V_{i}$ satisfies for all $r\gg 1$%
\begin{equation}
rV_{i}^{\prime }(r)\lesssim V_{i}(r),  \label{rV'}
\end{equation}%
then, for all $r\gg 1,$%
\[
\Lambda (B(o,r))\simeq V_{n}(r)h_{n}(r).
\]
\end{theorem}

Hence, in this case the Poincar\'{e} constant $\Lambda (B(o,r))$ is again
determined by the second largest volume function $V_{n}$
but the formula also involves the associated function $h_n$.

\begin{remark}
To obtain the upper bound of $\Lambda (B(o, r))$ in
 (\ref{PC upper non-parab}) and (\ref{PC upper parab}), 
we first prove the upper bound of $\Lambda(B(o,r), B(o, \kappa r))$ 
for some $\kappa>1$ (See section \ref{SecUBPoincare}). 
To reduce $\kappa>1$ to $\kappa =1$, we need an additional argument presented 
in Section \ref{section Whitney}.
\end{remark}

\begin{example}\label{ex11}
Assume that, under the hypotheses of Theorem \ref{main parab}, all $M_{i}$
have volume functions (\ref{aibi}) as in Example \ref{Exaibi}. Define on the
set of all pair $\left( \alpha _{i},\beta _{i}\right) $ the lexicographical
order $\succeq $, that is, 
\[
\left( \alpha _{i},\beta _{i}\right) \succeq \left( \alpha _{j},\beta
_{j}\right)
\]%
if $\alpha _{i}> \alpha _{j}$ or $\alpha _{i}=\alpha _{j}$ and $\beta
_{i}\geq \beta _{j}.$ Then we have 
\[
V_{i}\gtrsim V_{j}\Leftrightarrow \left( \alpha _{i},\beta _{i}\right)
\succeq \left( \alpha _{j},\beta _{j}\right) .
\]%
Clearly, each of the function (\ref{aibi}) satisfies (\ref{rV'}). By Theorem %
\ref{main parab}, we conclude that the Poincar\'{e} constant on $M$ is
determined by the second largest pair $\left( \alpha _{n},\beta _{n}\right) $%
, that is 
\[
\Lambda (B(o,r))\simeq V_{n}(r)h_{n}(r)\simeq \left\{ 
\begin{array}{ll}
r^{\alpha _{n}}(\log r)^{\beta _{n}} & \mbox{if }(\alpha _{n},\beta
_{n})\succ (2,1), \\ 
r^{2}\log r\log \log r & \mbox{if }(\alpha _{n},\beta _{n})=(2,1), \\ 
r^{2}\log r & \mbox{if }\alpha _{n}=2,\beta _{n}<1, \\ 
r^{2} & \mbox{if }\alpha _{n}<2.%
\end{array}%
\right.
\]
\end{example}

\begin{corollary}
Under the hypotheses of Theorem \emph{\ref{main parab}}, if 
\begin{equation}
|I_{super}|+|I_{middle}|\leq 1,  \label{I+I}
\end{equation}%
then
\[
\Lambda \left( B(o,r)\right) \lesssim r^{2},
\]%
that is,
\[
\lambda (B(o,r)) \gtrsim \frac{1}{r^2}
\]
for all large $r$.  
Consequently, $M$ satisfies {\rm{(PI)}}.
\end{corollary}

\begin{proof}
Indeed, under the hypothesis (\ref{I+I}) the second largest end $M_{n}$
belongs to $I_{sub}.$ By Definition \ref{DefCOE}, for a subcritical end we
have%
\[
h_{n}(r)\simeq r^{2}/V_{n}(r)
\]%
whence%
\[
\Lambda (B(o,r))\lesssim V_{n}\left( r\right) h_{n}\left( r\right) \simeq
r^{2}.
\]%
Note also that if $B(x,\kappa r)\subset E_{i}$ then, by (PI) on $M_{i}$,
we have 
\begin{equation}
\Lambda (B(x,r),B(x,\kappa r))\leq Cr^{2}.  \label{NC}
\end{equation}%
Hence, using the terminology of \cite[Sect. 4]{G-SC stability}, 
(PI) holds for \emph{anchored} and \emph{remote} balls in $M.$
By \cite[Prop. 4.2]{G-SC stability}, (PI) holds for all balls
in $M$.
\end{proof}

For some applications it is desirable to get rid of the hypothesis (COE).
We state the target result here as a conjecture.

\begin{conjecture}
Let $M=\#_{i\in I}M_{i}$ be a manifold with nice ends. 
Assume also that each parabolic end satisfies {\rm{(RCA)}}.
For any $r\gg 1$, define $m=m\left( r\right) $ and $n=n\left( r\right) $ so
that%
\[
V_{m}(r)h_{m}(r)\gtrsim V_{n}(r)h_{n}(r)\gtrsim V_{i}\left( r\right)
h_{i}\left( r\right) \ \ \mbox{for all }i\neq n,m.
\]%
Then%
\[
\Lambda (B(o,r))\lesssim V_{n}(r)h_{n}(r).
\]
\end{conjecture}

Obviously, Theorems \ref{main non-parab} and \ref{main parab} support this
conjecture. A typical case that is not covered by these theorems is when $%
M=M_{1}\#M_{2}$ where both volume functions $V_{1},V_{2}$ are close to
the critical case $V(r)=r^{2}$ but not ordered in the sense of $\gtrsim $.
See \cite{GISfluc} for such examples.

\subsection{Heat kernel estimates on manifolds with  nice ends}

Our strategy of the proof of Theorems \ref{main non-parab} and \ref{main
parab} is inspired by the argument of Kusuoka-Stroock \cite{KS} (see also 
\cite{SC LNS}), where (PI) was deduced from the heat kernel
estimates (\ref{LY}).

Let $M=\#_{i\in I}M_{i}$ be a manifold with nice ends as above. We obtain the
estimates of the Poincar\'{e} constant on $M$ by using heat kernel bounds on 
$M$. In the case when $M$ is non-parabolic, matching upper and lower
estimates of the heat kernel on $M$ were obtained in \cite{G-SC ends}. To
state this result, let us introduce the following notation: 
\[
\widetilde{V}_{i}:=V_{i}h_{i}^{2},
\]%
where $V_{i}\left( r\right) $ and $h_{i}\left( r\right) $ were defined in (%
\ref{Vi}) and (\ref{hi}), respectively.

\begin{theorem}
\label{non-parabolic} \label{G-SC on diagonal}(\cite[Cor. 6.8]{G-SC ends}, 
\cite{G-SC FK}) Let $M=\#_{i\in I}M_{i}$ be a manifold with nice  ends
 where  each parabolic $M_{i}$ satisfies 
{\rm{(RCA)}}. If $M$ is non-parabolic then, for all $t>0$, we have 
\[
p(t,o,o)\simeq \frac{1}{\min_{i\in I}\widetilde{V}_{i}(\sqrt{t})}.
\]
\end{theorem}

In particular, when all $M_{i}$ are non-parabolic then $\widetilde{V_{i}}%
\simeq V_{i}$ and we obtain%
\[
p(t,o,o)\simeq \frac{1}{\min_{i\in I}V_{i}(\sqrt{t})},
\]%
that is, $p\left( t,o,o\right) $ is determined by the end with the smallest
volume function.

In the case when $M$ is parabolic, similar estimates were obtained in \cite%
{GIS}, however, with certain restriction on the volume functions $%
V_{i}\left( r\right) $ near the critical case $V_{i}(r)\simeq r^{2}$. For
example, the result of \cite{GIS} includes $V_{i}\left( r\right) \simeq
r^{\alpha }$ with $\alpha \leq 2$ but does not include $V_{i}\left( r\right)
\simeq r^{2}\log r$ or $r^{2}/\log r$.

In this paper, we prove matching upper and lower heat kernel bounds on
parabolic manifolds with ends under weaker restrictions than in \cite{GIS}.

\begin{definition}
\label{DefRegular}A parabolic end $M_{i}$ is called \emph{regular} if it
satisfies (\ref{middle}) with positive exponents $\gamma _{1}$ and $\gamma
_{2}$ such that $2\gamma _{1}+\gamma _{2}<2,$ that is, 
\begin{equation}
c\left( \frac{R}{r}\right) ^{2-\gamma _{2}}\leq \frac{V_{i}(R)}{V_{i}(r)}%
\leq C\left( \frac{R}{r}\right) ^{2+\gamma _{1}}\ \ \mbox{for all }1\leq
r\leq R.  \label{regular}
\end{equation}
\end{definition}

For example, if $M$ satisfies (COE) (see Definition \ref%
{DefCOE}) then any $M_{i}$ with $i\in I_{middle}$ is regular.

\begin{definition}[DOE]
\label{DOE}We say that $M=\#_{i\in I}M_{i}$ has an end that dominates in the
order\emph{\ }(shortly {\rm{(DOE)}}) if there exists $l \in I$ such that, for all $%
i\in I$ 
\begin{equation}
V_{l }\gtrsim V_{i}~\mbox{ and }~\widetilde{V}_{l }\lesssim \widetilde{V}_{i}.
\label{DOE2}
\end{equation}
\end{definition}
\begin{remark}
If a manifold $M=\#_{i \in I} M_i$ satisfies {\rm{(COE)}} with $I_{super}=\emptyset$ and 
$I_{middle}\neq \emptyset$, then $M$ admits {\rm{(DOE)}} with $l \in I_{middle}$ and
 for all $r\gg 1$
\[
V_{l} (r) \simeq V_m(r),
\]
where $m=m(r)$ is the index of the largest end defined in (\ref{largest end}). 
See Lemma \ref{Lem2} for details.
\end{remark}
The next theorem is our main result regarding heat kernel estimates.

\begin{theorem}
\label{main heat} \label{main1}\label{main} Let $M=\#_{i\in I}M_{i}$ be a
manifold with parabolic nice ends, where each end satisfies {\rm{(RCA)}}
and is regular or subcritical. 
If  there exists at least one non-subcritical regular end, 
assume also that $M$ admits {\rm{(DOE)}}.
Then, for all $t>0$, 
\begin{equation}
p(t,o,o)\simeq  \frac{1}{V_m(\sqrt{t})},
\label{ptoo}
\end{equation}
where $m$ is defined in (\ref{largest end}). 
\end{theorem}
Hence, in the situation covered by this theorem, the long time behavior of
the heat kernel is determined by the \emph{largest} volume function, in
contrast to the case of non-parabolic ends where, as we have seen above, $%
p\left( t,o,o\right) $ is determined by the \emph{smallest} volume function.
The estimate (\ref{ptoo}) implies sharp matching upper and lower bounds for $%
p\left( t,x,y\right) $ for all $x,y\in M$ by means of the gluing techniques
of \cite{G-SC ends} (see also Theorem \ref{off-diagonal} below). 
\begin{example}
\label{example heat}
Let $M=\#_{i \in I} M_i$ be a manifold with nice ends, where each parabolic $M_i$ satisfies {\rm{(RCA)}}. 
Assume that the volume growth function $V_i(r)$ of $M_i$ satisfies 
for some $\alpha_i >0$ 
\[
V_i(r) \simeq r^{\alpha_i}, ~~(i \in I, ~ r \gg 1).
\]
Then Theorems \ref{non-parabolic} and \ref{main heat} imply that  for all $t>0$
\[
p(t, o, o) \simeq \frac{1}{t^{\alpha/2}},
\]
where $\alpha =2+\min_{i\in I}|\alpha_i-2|$ so that 
$p(t,o,o)$ is determined by $\alpha_i$ nearest to $2$, unlike the Poincar\'e constant 
$\Lambda (B(o,r))$.
\end{example}

\begin{example}Let $M=\#_{i \in I} M_i$ be as above.
Assume that $V_i(r)$ satisfies for some $\beta_i  \in \mathbb{R}$
\[
V_i(r) \simeq r^2 \left( \log r \right)^{\beta_i}, ~~(i \in I, ~ r \gg 1).
\]
Then Theorems \ref{non-parabolic} and \ref{main heat} imply that  for all $t>0$
\[
p(t, o, o) \simeq 
\left\{
\begin{array}{ll} 
\frac{1}{t \left( \log t \right)^{\beta}} & \mbox{if } \beta_i \neq 1 \mbox{ for all } i \in I , \\
\frac{1}{t \left( \log t \right) \left( \log \log t \right)^2} &\mbox{if }
 \beta_i=1\mbox{ for some } i \in I ,
\end{array}
\right.
\]
where $\beta =1+\min_{i \in  I}|\beta_i-1|$ so that 
$p(t,o,o)$ is determined by $\beta_i$ nearest to $1$.
\end{example}

It seems that the condition (DOE)  in Theorem \ref{main heat} is technical and 
we conjecture, that in general the following is true:
\begin{conjecture}
\label{conj heat} Let $M=\#_{i\in I}M_{i}$ be a manifold with  nice ends, 
and each parabolic end also satisfies 
{\rm{(RCA)}}. Then, for all $t>0$, we have 
\begin{equation}
p(t,o,o)\simeq \frac{\min_{i\in I}h_{i}^{2}(\sqrt{t})}{\min_{i\in I}%
\widetilde{V}_{i}(\sqrt{t})}.  \label{heat conj}
\end{equation}
\end{conjecture}


\begin{remark}Regarding condition {\rm{(DOE)}}, observe that
 the first condition $V_{l }\gtrsim V_{i}$ in (\ref{DOE2}) does
not imply in general the second condition $\widetilde{V}_{l }\lesssim 
\widetilde{V}_{i}$. However, if all ends are regular and there exists
 $\eta \geq 2\gamma _{1}$ such that for all $i$ 
\begin{equation}
\frac{V_{l }(r)}{V_{i}(r)}\geq Cr^{\eta }~~\mbox{for all }r\geq 1,
\label{delta}
\end{equation}%
then $\widetilde{V}_{l }\lesssim \widetilde{V}_{i}$ is satisfied. Indeed, by
the regularity of $V_{l }$ and  (\ref{hi}), we obtain 
\[
h_{l }(r)=1+\frac{1}{V_{l }\left( r\right) }\int_{1}^{r}\frac{V_{l }\left(
r\right) }{V_{l }\left( s\right) }sds\lesssim 1+\frac{1}{V_{l }\left( r\right) 
}\int_{1}^{r}\left( \frac{r}{s}\right) ^{2+\gamma _{1}}sds\lesssim \frac{%
r^{2+\gamma _{1}}}{V_{l }\left( r\right) },
\]%
which implies%
\[
\widetilde{V}_{l }(r)=V_{l}\left( r\right) h_{l }^{2}\left( r\right) \lesssim 
\frac{r^{4+2\gamma _{1}}}{V_{l }(r)}.
\]%
Since $h_{i}(r)\gtrsim \frac{r^{2}}{V_{i}(r)}$, we obtain 
\[
\widetilde{V}_{i}(r)=V_{i}\left( r\right) h_{i}^{2}\left( r\right) \gtrsim 
\frac{r^{4}}{V_{i}(r)}.
\]%
Hence, (\ref{delta}) implies that 
\[
\frac{\widetilde{V}_{l }(r)}{\widetilde{V}_{i}(r)}\lesssim \frac{V_{i}(r)}{%
V_{l }(r)}r^{2\gamma _{1}}\lesssim r^{2\gamma _{1}-\eta }\lesssim 1,
\]%
as was claimed.
\end{remark}

\subsection{Structure of the paper}

In Section \ref{SecExamples} we give examples of application of Theorems \ref%
{main non-parab} and \ref{main parab} in the case when all ends are model
manifolds.

In Section \ref{SecHK}, we first survey the previous results of \cite{G-SC
ends} and \cite{GIS} on heat kernel estimates on manifolds with ends and
then prove Theorem \ref{main1}.

%
In Section \ref{SecPI} we give the proofs of Theorems \ref{main non-parab}
and \ref{main parab}. For that, we first obtain in Section \ref%
{SecLBDirichlet} a lower bound for the Dirichlet heat kernel in balls (Lemma %
\ref{Dirichlet}) using the two-sided off-diagonal bounds of the heat kernel $%
p\left( t,x,y\right) $ on $M$ that follow from Theorem \ref{main1}. By means
of this estimate, we obtain in Section \ref{SecUBPoincare} an upper bound
for $\Lambda (B(o,r),B(o,\kappa r))$ for some large enough $\kappa >1$. 
We use an additional argument to reduce $ \kappa $ to $1$ stated in
Section \ref{section Whitney}. The lower bound for $\Lambda (B(o,r))$ is proved in Section %
\ref{SecLBPoincare}. 

In Section \ref{section Whitney}, we obtain a general upper bound of the Poincar\'e 
constant $\Lambda(B(o,r))$ by using a collection of $\Lambda(B(x,s), B(x, \kappa s))$, 
where $B(x,s) \subset B(o,r)$ and $\kappa \geq 1$.

\section{Model manifold and examples}

\label{SecExamples}
\setcounter{equation}{0}
 Let $\psi $ be a smooth positive
function on $\left( 0,+\infty \right) $ such that such $\psi \left( r\right)
=r$ for $r<1$. Fix an integer $N\geq 2$ and consider in $\mathbb{R}^{N}$ a
Riemannian metric $g_{\psi }$ given in the polar coordinates $(r,\theta )\in
(0,\infty )\times \mathbb{S}^{N-1}$ by 
\[
g_{\psi }=dr^{2}+\psi (r)^{2}d\theta ^{2}.
\]%
In a punctured neighborhood of the origin $o\in \mathbb{R}^{N}$, this metric
coincides with the Euclidean one and, hence, extends to the full
neighborhood of $o$, so that $g_{\psi }$ is defined on the entire $\mathbb{R}%
^{N}$. The Riemannian manifold $\left( \mathbb{R}^{N},g_{\psi }\right) $ is
called a model manifold. For example, if $\psi \left( r\right) =r$ then $%
g_{\psi }$ is the canonical Euclidean metric.

Let us assume in addition that, for some $C>0$ and all $r\gg 1$,%
\begin{equation}
\left\{ 
\begin{array}{l}
\sup\limits_{\lbrack r,2r]}\psi \leq C\inf\limits_{[r,2r]}\psi , \\ 
\psi (r)\leq Cr, \\ 
\displaystyle{\int_{0}^{r}}\psi ^{N-1}(s)ds\leq Cr\psi ^{N-1}(r).%
\end{array}%
\right.  \label{psi3}
\end{equation}%
Since for any bounded range $r\in \left( 0,r_{0}\right) $ the condition (\ref%
{psi3}) is trivially satisfied, we obtain by \cite[Prop. 4.10]{G-SC
stability} that $\left( \mathbb{R}^{N},g_{\psi }\right) $ satisfies (PI)
and (VD). It is also obvious that $\left( \mathbb{R}^{N},g_{\psi }\right) $
satisfies (RCA) with respect to the origin $o$ because the geodesic balls $%
B\left( o,r\right) $ coincide with the Euclidean balls.

Let us reformulate conditions (\ref{psi3}) in terms of the volume $V\left(
r\right) $ of a ball $B(o,r)$ on $\left( \mathbb{R}^{N},g_{\psi }\right) $:%
\begin{equation}
V(r)=V\left( o,r\right) =\omega _{N}\int_{0}^{r}\psi ^{N-1}(s)ds.
\label{V(r)}
\end{equation}

\begin{lemma}
\label{Lempsi3}The conditions (\ref{psi3}) are equivalent to the following
conditions to be satisfied for all $r\gg 1$:
\begin{equation}
\begin{array}{l}
V(r)\leq Cr^{N} \\ 
V\left( r\right) \simeq rV^{\prime }\left( r\right) .%
\end{array}
\label{V2}
\end{equation}
\end{lemma}

\begin{proof}
The inequality $V\left( r\right) \leq Cr^{N}$ follows from $\psi \left(
r\right) \leq Cr$ and (\ref{V(r)}). The third condition in (\ref{psi3}) is
equivalent to%
\[
V\left( r\right) \leq CrV^{\prime }\left( r\right)
\]%
that is the upper bound in the second condition in (\ref{V2}). To obtain a
similar lower bound, observe that the first condition in (\ref{psi3}) is
equivalent to%
\begin{equation}
V^{\prime }\left( s\right) \simeq V^{\prime }\left( r\right) \ \ \mbox{for
all }s\in \left[ \frac{1}{2}r,r\right] ,  \label{V'sr}
\end{equation}%
which implies%
\[
V\left( r\right) \geq \int_{r/2}^{r}V^{\prime }\left( s\right) ds\simeq
rV^{\prime }\left( r\right) .
\]%
Let us now prove that (\ref{V2}) implies (\ref{psi3}). Clearly, the third
condition in (\ref{psi3}) follows from 
\begin{equation}
V\left( r\right) \simeq rV^{\prime }\left( r\right) .  \label{VV'}
\end{equation}%
The conditions (\ref{V2}) imply the second condition in (\ref{psi3}) as
follows: 
\begin{equation}
\psi \left( r\right) ^{N-1}=\omega _{N}^{-1}V^{\prime }\left( r\right)
\simeq \frac{V\left( r\right) }{r}\leq Cr^{N-1}.  \label{psi cond}
\end{equation}%
Finally, let us prove the first condition in (\ref{psi3}), or, equivalently,
(\ref{V'sr}). In the view of (\ref{VV'}) and (\ref{psi cond}) it suffices to
prove that $V$ satisfies the doubling property, that is,%
\begin{equation}
V\left( 2r\right) \leq CV\left( r\right) .  \label{V2r}
\end{equation}%
Indeed, applying again (\ref{VV'}), we obtain%
\[
\ln V\left( 2r\right) -\ln V\left( r\right) =\int_{r}^{2r}\frac{V^{\prime
}\left( t\right) }{V\left( t\right) }dt\simeq \int_{r}^{2r}\frac{dt}{t}=\ln
2,
\]%
whence (\ref{V2r}) follows.
\end{proof}

In view of Lemma \ref{Lempsi3}, it will be more convenient to describe a
model manifold in terms of the volume function $V\left( r\right) $ rather
than $\psi \left( r\right) $. Hence, we denote the model manifold $\left( 
\mathbb{R}^{N},g_{\psi }\right) $ shortly by $\mathcal{M}_{V}$. 

Given $k$ model manifolds $\mathcal{M}_{V_{1}},....,\mathcal{M}_{V_{k}}$
satisfying (\ref{V2}) (and, hence, {\rm{(VD)}} and {\rm{(PI)}}), consider their
connected sum%
\[
M=\mathcal{M}_{V_{1}}\#...\#\mathcal{M}_{V_{k}}.
\]%
Assume that the standing assumptions of either Theorem \ref{main non-parab}
or Theorem \ref{main parab} are satisfied, that is, either all ends are
non-parabolic or all ends are (COE). Then, Theorems \ref{main non-parab} and %
\ref{main parab} yield that%
\[
\Lambda (B(o,r))\simeq V_{n}(r)h_{n}(r),
\]%
that is,
\[
\lambda (B(o,r))\simeq \frac{1}{V_{n}(r)h_{n}(r)},
\]%
where $n$ is the index of the second largest end.

\begin{example}
Assume that, for all $r\gg 1$,%
\begin{equation}
V\left( r\right) =r^{\alpha }\prod_{j=1}^{J}(\log _{[j]}r)^{\beta (j)},
\label{Vab}
\end{equation}%
where $\alpha >0$, $\beta (1),\ldots ,\beta (J )\in \mathbb{R}$ and $\log
_{[j]}r$ is the $j$-times iterated logarithm. Clearly, in this case $%
\mathcal{M}_{V}$ is parabolic if and only if 
\begin{equation}
(\alpha ,\beta (1),\ldots ,\beta (J ))\preceq (2,1,\ldots ,1),
\label{abparab}
\end{equation}%
where $\preceq $ denotes the lexicographical order. Also, it is easy to see
that $V\left( r\right) $ satisfies (\ref{V2})  and hence (PHI) if and only if 
\begin{equation}
(\alpha ,\beta (1),\ldots ,\beta (J ))\preceq (N,0,\ldots ,0).  \label{ab}
\end{equation}
Note for comparison that the function $V\left( r\right) =\left( \log
r\right) ^{\beta }$ with $\beta >0$ does not satisfy the second condition in
(\ref{V2}).
Some examples of model manifolds with $V\left( r\right) =r^{\alpha }$ are
shown on Fig. \ref{figure: model2}. 
\begin{figure}[tbph]
\begin{center}
\scalebox{1.2}{
\includegraphics{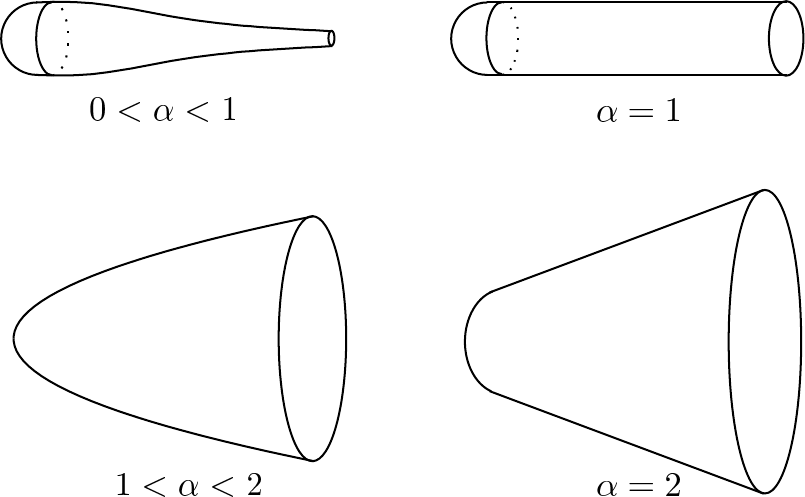} 
}
\end{center}
\par
\caption{Model manifolds $\mathcal{M}_{V}$ of dimension $N=2$ with $V(r)=r^{%
\protect\alpha }$, $r\gg 1.$}\label{figure: model2}
\end{figure}
\end{example}

Assume now that for each $i$ and all $r \gg 1$
\begin{equation}
V_{i}\left( r\right) =r^{\alpha _{i}}\prod_{j=1}^{J }(\log
_{[j]}r)^{\beta _{i}(j)},  \label{Vex}
\end{equation}%
where each $(J +1)$-tuple $\left( \alpha _{i},\beta _{i}(1),\ldots
,\beta _{i}(J )\right) $ satisfies (\ref{ab}). The condition (COE) holds
in this case with the following decomposition of the index set $I$: 
\[
i\in \left\{ 
\begin{array}{ll}
I_{super} &\mbox{if }\alpha _{i}>2, \\ 
I_{sub} & \mbox{if }\alpha _{i}<2, \\ 
I_{middle} &\mbox{if\ }\alpha _{i}=2.%
\end{array}%
\right.
\]%
Let us compute $V_{i}h_{i}$. If $V_{i}$ is non-parabolic, that is, if 
\[
(\alpha _{i},\beta _{i}(1),\ldots ,\beta _{i}(J ))\succ (2,1,\ldots ,1),
\]%
then by (\ref{hi}) $h_{i}\simeq 1$ and, hence, $V_{i}h_{i}\simeq V_{i}.$

Let now $V_{i}$ be parabolic, that is, 
\begin{equation}
\left( \alpha _{i},\beta _{i}(1),\ldots ,\beta _{i}(J )\right) \preceq
(2,1,\ldots ,1).  \label{aipar}
\end{equation}%
If $\alpha _{i}<2$ then $V_{i}$ is subcritical and, hence, 
\begin{equation}
V_{i}h_{i}\left( r\right) \simeq r^{2}.  \label{r2}
\end{equation}%
In the case $\alpha _{i}=2$, we denote by $J _{i}$ the smallest index $%
j=1,...,J $ such that $\beta _{i}\left( j\right) <1;$ if such $j$ does
not exists then take $J _{i}=J +1.$ The parabolicity of $V_{i}$
implies that 
\[
\beta _{i}(j)=1\ \ \mbox{for all }j<J _{i}.
\]%
With this notation we obtain, for all $r\gg 1$, 
\[
h_{i}(r)\simeq \left\{ 
\begin{array}{ll}
r^{2-\alpha _{i}}\prod_{j=1}^{J }\left( \log _{[j]}r\right) ^{-\beta
_{i}(j)} & \mbox{if }\alpha _{i}<2 \\ 
\left( \log _{[J _{i}]}r\right) \prod_{j=J _{i}}^{J }\left( \log
_{[j]}r\right) ^{-\beta _{i}(j)} & \mbox{if }\alpha _{i}=2.%
\end{array}%
\right.
\]%
Combining this, (\ref{r2}) and (\ref{Vex}), we obtain 
\begin{equation}
V_{i}\left( r\right) h_{i}\left( r\right) \simeq \left\{ 
\begin{array}{ll}
r^{2} & \mbox{if }\alpha _{i}<2, \\ 
r^{2}\prod_{j=1}^{J _{i}}\log _{[j]}r & \mbox{if }\alpha _{i}=2%
\end{array}%
\right.  \label{ViHi}
\end{equation}%
and 
\begin{equation}
\widetilde{V}_{i}(r)\simeq \left\{ 
\begin{array}{ll}
r^{4-\alpha _{i}}\prod_{j=1}^{J }\left( \log _{[j]}r\right) ^{-\beta
_{i}(j)} & \mbox{if }\alpha _{i}<2, \\ 
r^{2}\prod_{j=1}^{J _{i}-1}\left( \log _{[j]}r\right) (\log _{[J
_{i}]}r)^{2}\prod_{j=J _{i}}^{J }\left( \log _{[j]}r\right) ^{-\beta
_{i}(j)} & \mbox{if }\alpha _{i}=2.%
\end{array}%
\right.  \label{Vtilex}
\end{equation}

Theorems \ref{main non-parab}, \ref{main parab}, \ref{main heat} and \ref%
{non-parabolic} yield the following. Let $m$ be the index for the largest
volume function and let $n$ be the index for the \emph{second largest} volume
function, that is, for all $i\neq m, n,$ 
\[
(\alpha _{m},\beta _{m}(1),\ldots ,\beta _{m}(J ))\succeq (\alpha
_{n},\beta _{n}(1),\ldots ,\beta _{n}(J ))\succeq (\alpha _{i},\beta
_{i}(1),\ldots ,\beta _{i}(J )).
\]

\emph{Case 1}. Let $M$ be non-parabolic. Then 
\begin{equation}
p(t,o,o)\simeq \frac{1}{\min_{i\in I}\widetilde{V}_{i}(\sqrt{t})},
\label{model heat1}
\end{equation}%
that is, $p(t,o,o)$ is determined by the smallest function $\widetilde{V}%
_{i}(\sqrt{t})$ (see further examples and comments in \cite{G-SC ends}).
Note that it is not easy to give an explicit expression for 
$\min_{ i \in  I} \widetilde{V}_i$ because the form of the result varies depending on the exact nature of 
$(\alpha_i, \beta_i(1), \ldots ,\beta_i(J)) , 1\le  i\le k$. 
Of course, in any particular example, one can compute $\min_{i \in I} \widetilde{V}_i$ explicitly.

\emph{Case 2}. Let $M$ be parabolic. Then 
\begin{equation}
p(t,o,o)\simeq \frac{1}{V_{m}(\sqrt{t})}\simeq \frac{1}{t^{\alpha
_{m}/2}\prod_{j=1}^{J }\left( \log _{[j]}t\right) ^{\beta _{m}(j)}}.
\label{model heat2}
\end{equation}
If $V_i(r) \simeq r^{\alpha_i}$, then the above estimates in (\ref{model heat1}) and (\ref{model heat2}) 
imply that
\[
p(t,o,o) \simeq \frac{1}{t^{\alpha/2}},
\]
where
\[
\alpha=2+\min_{i \in I} |\alpha_i-2|
\]
so that $p(t,o,o)$ is determined by the volume growth function nearest to $r^2$.

In both case 1 and case 2, the Poincar\'{e} constant $\Lambda \left( B\left(
o,r\right) \right) $ is determined by the \emph{second largest} volume function $%
V_{n}(r)$, namely, for all $r\gg 1$ 
\[
\Lambda (B(o,r))\simeq V_{n}(r)h_{n}(r).
\]%
If $V_{n}$ is non-parabolic, then $V_{n}h_{n}\simeq V_{n}$, whence we obtain 
\[
\Lambda (B(o,r))\simeq r^{\alpha _{n}}\prod_{j=1}^{J }\left( \log
_{[j]}r\right) ^{\beta _{n}(j)},
\]%
that is, 
\[
\lambda (B(o,r))\simeq r^{-\alpha _{n}}\prod_{j=1}^{J }\left( \log
_{[j]}r\right) ^{-\beta _{n}(j)}.
\]%

If $V_{n}$ is parabolic, then by (\ref{ViHi}) we obtain 
\begin{equation}
\Lambda (B(o,r))\simeq \left\{ 
\begin{array}{ll}
r^{2} & \mbox{if }\alpha _{n}<2, \\ 
r^{2}\prod_{j=1}^{J _{n}}\log _{[j]}r & \mbox{if }\alpha _{n}=2,%
\end{array}%
\right.  \label{rigidity}
\end{equation}%
that is,
\begin{equation}
\lambda (B(o,r))\simeq \left\{ 
\begin{array}{ll}
\frac{1}{r^{2}} & \mbox{if }\alpha _{n}<2, \\ 
\frac{1}{r^{2}\prod_{j=1}^{J _{n}}\log _{[j]}r }& \mbox{if }\alpha _{n}=2.%
\end{array}%
\right.  
\end{equation}%
For example, (\ref{rigidity}) holds provided at most one of the ends of $M$
is non-parabolic, because in this case the second largest$\ V_{n}$ is
parabolic. We see from (\ref{rigidity}) that in this case the estimate of
the Poincar\'{e} constant exhibits a certain rigidity -- it does not depend
on the exponents $\alpha _{i}$ and $\beta _{i}\left( j\right) ,$ although
it does depend on $J _{n}.$

\section{Heat kernel estimates (central estimates)}

\label{SecHK}
\setcounter{equation}{0}
This section is devoted to obtain
heat kernel estimates on manifolds with nice  ends, which play a key role in the
proof of Theorem \ref{main parab}. The main result here is Theorem \ref{main}
which provides two-sided matching estimates of $p\left( t,o,o\right) $ on
parabolic manifolds with nice ends. This result is of interest by itself
independently of the application to bounding the Poincar\'{e} constant. Theorem \ref{main}
implies the off-diagonal estimates of $p\left( t,x,y\right) $ as in Theorem %
\ref{off-diagonal}.

\subsection{Known results}

\label{Secsubcritical}Let us first recall some known results. As before,
consider a manifold with ends $M=M_{1}\#...\#M_{k}$ where each $M_{i}$ is nice, namely 
$M_i$ is a geodescially complete, weighted manifold that satisfies (PI) and (VD) (see Definition \ref{nice}).
Besides, if $M_{i}$ is parabolic then we assume that it satisfies (RCA). 
Fix a reference point $o_{i}\in M_{i}$, set $V_{i}\left(
r\right) =V_{i}\left( o_{i},r\right) $ and define $h_{i}\left( r\right) $ by
(\ref{hi}).

In the case when $M$ is non-parabolic the matching upper and lower bounds of 
$p\left( t,o,o\right) $ are provided by Theorem \ref{non-parabolic}. Let us
discuss the case when $M$ is parabolic.

Recall that an end $M_{i}$ is called subcritical if, for all $r\gg 1$,%
\begin{equation}
h_{i}(r)\simeq \frac{r^{2}}{V_{i}(r)}  
\end{equation}%
(see Definition \ref{def subcritical}). For example, if $V_{i}\left( r\right) \simeq
r^{\alpha }$ then \thinspace $M_{i}$ is subcritical if and only if $\alpha
<2 $. Conversely, if $M_{i}$ is subcritical, then there exists $\delta >0$
such that 
\begin{equation}
V_{i}(r)\lesssim r^{2-\delta }.  \label{subcritical upper}
\end{equation}%
Indeed, (\ref{def subcritical}) implies that, for some $\delta >0$, 
\[
(\log h_{i}(r))^{\prime }\geq \frac{\delta }{r},
\]%
which implies upon integration that $h_{i}\left( r\right) \gtrsim r^{\delta
} $, which together with (\ref{def subcritical}) yields (\ref{subcritical upper}%
).

\begin{definition}
An end $M_{i}$ is called \emph{critical} if 
$$V_{i}(r)\simeq r^{2}.$$ 
Note
that all critical and subcritical ends are parabolic.
\end{definition}

\begin{theorem}
\label{GIS} (\cite[Theorem 2.1]{GIS}) Let $M=M_{1}\#\cdots \#M_{k}$ where
each end $M_{i}$ is either critical or subcritical. Then, for all $t>0$, we
have%
\[
p(t,o,o)\simeq \frac{1}{V_{m}(\sqrt{t})},
\]%
where $m=m(r)$ is the index of the largest volume function at scale $r$.
\end{theorem}

Our main goal in this section is to prove Theorem \ref{main} stated in
Introduction that provides the same estimates of $p\left( t,o,o\right) $ but
for a more general class of parabolic manifolds with ends (see Section \ref%
{SecExamples}).

Let us first explain how the estimates of $p(t,o,o)$ lead to the estimates
of $p\left( t,x,y\right) $ for all $x,y\in M.$ For $r,t>0$, set 
\begin{eqnarray*}
H_{i}\left( r,t\right) &=&\frac{r^{2}}{V_{i}(r)h_{i}(r)}+\frac{1}{h_{i}(%
\sqrt{t})}\left( \int_{r}^{\sqrt{t}}\frac{sds}{V_{i}(s)}\right) _{+} \\
&=&\frac{r^{2}}{V_{i}(r)h_{i}(r)}+\frac{\left( h_{i}\left( \sqrt{t}\right)
-h_{i}\left( r\right) \right) _{+}}{h_{i}(\sqrt{t})}
\end{eqnarray*}%
and%
\[
D_{i}\left( r,t\right) =\frac{h_{i}(r)}{h_{i}(r)+h_{i}(\sqrt{t})}.
\]

The following result follows from \cite[Theorems 3.5 ]{G-SC ends}, \cite[%
Thms 4.4, 4.6]{G-SC hitting}, \cite[Thms 3.3, 4.9]{G-SC Dirichlet}, \cite[%
(5.1), (5.6), (5.7)]{GIS} and the local Harnack inequality near $o\in K$.
Set $\left\vert x\right\vert =d\left( x,K\right) +3.$

\begin{theorem}
\label{off-diagonal} Let $M=M_{1}\#\cdots \#M_{k}$ be a manifold with nice ends,
where  each parabolic end satisfies {\rm{(RCA)}}. Then, for all $x\in E_{i}$, $y\in E_{j}$, $t\gg 1$ we
have 
\begin{small}
\begin{eqnarray*}
p(t,x,y) &\asymp &\delta _{ij}\frac{1}{V_{i}(x,\sqrt{t})}%
D_{i}\left( \left\vert x\right\vert ,t\right) D_{j}\left( \left\vert
y\right\vert ,t\right) e^{-b\frac{d^{2}(x,y)}{t}} \\
&&\!\!\!\!\!\!\!\!\!\!\!\!\!\!\!\!\!\!\!\!+p(t,o,o)H_{i}\left( \left\vert x\right\vert ,t\right)
H_{j}\left( \left\vert y\right\vert ,t\right) e^{-b\frac{|x|^{2}+|y|^{2}}{t}}\\
&&\!\!\!\!\!\!\!\!\!\!\!\!\!\!\!\!\!\!\!\! +\int_{1}^{t}p(s,o,o)ds\left[ \frac{D_{i}\left( \left\vert
x\right\vert ,t\right) H_{j}\left( \left\vert y\right\vert ,t\right) }{V_{i}(%
\sqrt{t})h_{i}(\sqrt{t})}+\frac{D_{j}\left( \left\vert y\right\vert
,t\right) H_{i}\left( \left\vert x\right\vert ,t\right) }{V_{j}(\sqrt{t}%
)h_{j}(\sqrt{t})}\right] e^{-b\frac{|x|^{2}+|y|^{2}}{t}}. 
\end{eqnarray*}
\end{small}
\end{theorem}
This is a very general estimate.  
In the view of this result, the on-diagonal value $p(t,o,o)$ of the heat
kernel plays a key role in estimating of $p\left( t,x,y\right) .$
 One of the aim of this paper is to improve upon the existing results from \cite{GIS}.
See Theorem \ref{main}.

\subsection{Preliminary estimates}

A classical method of obtaining on-diagonal heat kernel bounds is to use a
Nash-type functional inequality, which gives a uniform upper bound for $%
p\left( t,x,x\right) $ for all $x\in M.$ Indeed, such an inequality can be
proved on manifold $M$ in the setting of Theorem \ref{off-diagonal} (see 
\cite{G-SC FK}) but in the case when $M$ is parabolic, this upper bound is
not optimal. In order to obtain an optimal upper bound, we have developed in 
\cite[Section 3]{GIS} a different method based on the \emph{integrated
resolvent kernel} $\gamma _{\lambda }(x)$ and its derivative $\dot{\gamma}%
_{\lambda }(x).$ However, in \cite[Section 3]{GIS} we could handle only
critical and subcritical ends (cf. Theorem \ref{GIS}). For the proof of
Theorem \ref{main}, we apply the method of \cite[Section 3]{GIS}, but with
significant improvements that allow us to handle much more general ends.

Let $(M,\mu )$ be a geodesically complete, non-compact weighted manifold.
For any $\lambda >0$, we define the \emph{resolvent operator} $G_{\lambda }$
acting on non-negative measurable functions $f$ on $M$ by 
\[
G_{\lambda }f(x)=\int_{M}\int_{0}^{\infty }e^{-\lambda t}p(t,x,z)f(z)dzdt.
\]%
We remark that $u=G_{\lambda }f$ is the minimal non-negative weak solution
of the equation 
\begin{equation}
\Delta u-\lambda u=-f  \label{Dirichlet problem}
\end{equation}%
(see \cite{G AMS}). Fix a compact set $K\subset M$. We define the functions $%
\gamma _{\lambda }$ and $\dot{\gamma}_{\lambda }$ on $M$ by 
\[
\gamma _{\lambda }(x):=G_{\lambda }1_{K}(x)=\int_{K}\int_{0}^{\infty
}e^{-\lambda t}p(t,x,z)dzdt
\]%
and 
\[
\dot{\gamma}_{\lambda }(x):=G_{\lambda }\gamma _{\lambda }=-\frac{\partial }{%
\partial \lambda }\gamma _{\lambda }(x)=\int_{K}\int_{0}^{\infty
}te^{-\lambda t}p(t,x,z)dzdt.
\]%
Fix a reference point $o\in K$. The following lemma follows from the
estimates in \cite[(3.10), (3.11)]{GIS}.

\begin{lemma}
There exist constants $C>0$ and $t_{0}>0$ depending on $K$ and such that,
for all $x\in K$ and $t\geq t_{0}$, 
\begin{eqnarray}
p(t,o,o)& \leq & \frac{C}{t}\gamma _{\frac{1}{t}}(x),  \label{p gamma} \\
p(t,o,o)& \leq &\frac{C}{t^{2}}\dot{\gamma}_{\frac{1}{t}}(x).
\label{p gamma dot}
\end{eqnarray}
\end{lemma}

As in \cite[Section 4]{GIS}, in order to estimate $\gamma _{\lambda }$ and $%
\dot{\gamma}_{\lambda }$, we will use the functions $\Phi _{\lambda
}^{\Omega }$ and $\Psi _{\lambda }^{\Omega }$ defined below. Denote by $%
p_{\Omega }^{D}(t,x,y)$ the Dirichlet heat kernel in an open set $\Omega
\subset M$. For any $\lambda >0$, define the resolvent operator $G_{\lambda
}^{\Omega }$ on non-negative measurable functions $f$ in $\Omega $ by 
\[
G_{\lambda }^{\Omega }f(x)=\int_{\Omega }\int_{0}^{\infty }e^{-\lambda
t}p_{\Omega }^{D}(t,x,z)f(z)dzdt,
\]%
so that $G_{\lambda }^{\Omega }f$ is the minimal non-negative weak solution
of the equation $\Delta u-\lambda u=-f$ in $\Omega $.

Now we define the functions $\Phi _{\lambda }^{\Omega }$ and $\Psi _{\lambda
}^{\Omega }$ in $\Omega $ by 
\begin{eqnarray*}
\Phi _{\lambda }^{\Omega }&=& \lambda G_{\lambda }^{\Omega }1, \\
\Psi _{\lambda }^{\Omega }&=& G_{\lambda }^{\Omega }(1-\Phi _{\lambda
}^{\Omega }).
\end{eqnarray*}%
These functions have certain probabilistic meaning. Indeed, denote by $%
(\{X_{t}\}_{t\geq 0},\{\mathbb{P}_{x}\}_{x\in M})$ Brownian motion on $M$.
For any open set $\Omega \subset M$, let $\tau _{\Omega }$ be the first exit
time of $X_{t}$ from $\Omega $, that is, 
\[
\tau _{\Omega }=\inf \{t>0:X_{t}\not\in \Omega \}.
\]%
Then, by \cite[(3.17), (3.25)]{GIS}, for any $x\in \Omega $, 
\begin{eqnarray}
\Phi _{\lambda }^{\Omega }(x)& =&\int_{0}^{\infty }\lambda e^{-\lambda t}%
\mathbb{P}_{x}(\tau _{\Omega }>t)dt,  \label{FiOm} \\
\Psi _{\lambda }^{\Omega }(x)& =& \int_{0}^{\infty }te^{-\lambda t}\partial
_{t}\mathbb{P}_{x}(\tau _{\Omega }\leq t)dt.  \label{PsiOm}
\end{eqnarray}%
Now let us assume that the manifold $M$ satisfies $\mathrm{(PI)}$, $\mathrm{(VD)}$ and $\mathrm{(RCA)}$%
. Let $A$ be a precompact open set containing%
\[
K_{\varepsilon }:=\{z\in M:d(K,z)<\varepsilon \}
\]%
for some $\varepsilon >0$. Define function $h\left( r\right) $ for $r>0$ by 
\begin{equation}
h(r)=1+\left( \int_{1}^{r}\frac{sds}{V(s)}\right) _{+},  \label{hdef}
\end{equation}%
where $V(r)=V(o,r)$. In the next lemmas, we obtain some new estimates of $%
\Phi _{\lambda }^{K^{c}}$ and $\Psi _{\lambda }^{K^{c}}$.

\begin{lemma}
There exists $c,\lambda _{0}>0$ depending on $K$ and $A$, and such that, for
all $\lambda \in \left( 0,\lambda _{0}\right) $, 
\begin{equation}
\inf_{\partial A}\Phi _{\lambda }^{K^{c}}\geq \frac{c}{h(\frac{1}{\sqrt{%
\lambda }})}.  \label{Phi}
\end{equation}%
%
%
%
%
%
%
%
%
%
%
%
%
%
%
%
%
%
%
\end{lemma}

\begin{proof}
By \cite[(4.47)]{G-SC hitting}, we have, for all $x\in \partial A$ and $%
t\geq t_{1},$ 
\[
\mathbb{P}_{x}(\tau _{K^{c}}>t)\approx \frac{h(|x|)}{h(\sqrt{t})}
\]%
where $t_{1}>0$ depends on $A$. Substituting this into (\ref{FiOm}) and
using the fact that $\inf_{\partial A}h\left( \left\vert x\right\vert
\right) >0$, we obtain, for all $\lambda <1/\left( 2t_{1}\right) $ and $x\in
\partial A$, 
\[
\Phi _{\lambda }^{K^{c}}(x)\geq \int_{1/\left( 2\lambda \right) }^{1/\lambda
}\lambda e^{-\lambda t}\frac{c}{h(\sqrt{t})}dt\geq \frac{c}{h(\frac{1}{\sqrt{%
\lambda }})}\int_{1/\left( 2\lambda \right) }^{1/\lambda }\lambda
e^{-\lambda t}dt\geq \frac{c}{h(\frac{1}{\sqrt{\lambda }})}.
\]
\end{proof}

\begin{remark}
If $M$ is critical, then the estimate in (\ref{Phi}) implies that 
\[
\inf_{\partial A}\Phi _{\lambda }^{K^{c}}\geq \frac{c}{\log \frac{1}{\sqrt{%
\lambda }}},
\]%
which coincides with \cite[(3.34)]{GIS}. If $M$ is subcritical, the estimate
in (\ref{Phi}) implies that 
\[
\inf_{\partial A}\Phi _{\lambda }^{K^{c}}\geq c\lambda V(\frac{1}{\sqrt{%
\lambda }}),
\]%
which is identical to \cite[(3.33)]{GIS}.
\end{remark}

Recall that $V\left( r\right) $ is regular if it satisfies (\ref{regular})
that is, 
\[
c\left( \frac{R}{r}\right) ^{2-\gamma _{2}}\leq \frac{V(R)}{V(r)}\leq
C\left( \frac{R}{r}\right) ^{2+\gamma _{1}}\ \ \mbox{for all }1\leq r\leq R
\]%
for some positive $\gamma _{1},\gamma _{2}$ such that $2\gamma _{1}+\gamma
_{2}<2.$

Set $\widetilde{V}(r)=V(r)h^{2}(r)$. Then we obtain the following upper
estimate of $\widetilde{V}(r)$ on manifold with regular volume function.

\begin{lemma}
Let $M$ be regular, that is, the volume function $V$ is regular. Then, for
all $1\leq r\leq R$,%
\begin{eqnarray}
\widetilde{V}(r)& \lesssim & r^{2+\gamma _{1}+\gamma _{2}}\lesssim r^{4},
\label{Vtil1} \\
\frac{\widetilde{V}(R)}{\widetilde{V}(r)}& \lesssim &\left( \frac{R}{r}%
\right) ^{2+2\gamma _{1}+\gamma _{2}}.  \label{Vtil2}
\end{eqnarray}
\end{lemma}

\begin{proof}
By (\ref{hdef}) and regularity, we have 
\begin{small}
\begin{equation}
V(r)h(r)=V\left( r\right) +\int_{1}^{r}\frac{V\left( r\right) }{V\left(
s\right) }sds\leq C\frac{V\left( r\right) }{V\left( 1\right) }%
+C\int_{1}^{r}\left( \frac{r}{s}\right) ^{2+\gamma _{1}}sds\leq C^{\prime
}r^{2+\gamma _{1}}.  \label{Vh}
\end{equation}%
\end{small}
Similarly, we have 
\[
h(r)=1+\int_{1}^{r}\frac{sds}{V\left( s\right) }\leq 1+\frac{1}{V\left(
1\right) }\int_{1}^{r}\frac{V\left( 1\right) }{V\left( s\right) }sds\leq
1+C\int_{1}^{r}\frac{sds}{s^{2-\gamma _{2}}}\leq C^{\prime }r^{\gamma _{2}},
\]%
whence (\ref{Vtil1}) follows.

Now we prove (\ref{Vtil2}). By the regularity of $V$ and 
\[
h\left( r\right) \geq c\frac{r^{2}}{V\left( r\right) },
\]%
we obtain 
\begin{small}
\begin{eqnarray*}
\frac{\widetilde{V}(R)}{\widetilde{V}(r)}&=& \frac{V(R)h^{2}\left( R\right) }{%
V(r)h^{2}\left( r\right) }=\frac{V(R)\left( h(r)+\int_{r}^{R}\frac{sds}{V(s)}%
\right) ^{2}}{V(r)h^{2}(r)}=\frac{V(R)}{V(r)}\left( 1+\frac{\int_{r}^{R}%
\frac{sds}{V(s)}}{h(r)}\right) ^{2} \\
&\leq & \frac{V(R)}{V(r)}\left( 1+C\frac{V(r)}{r^{2}}\frac{\int_{r}^{R}\frac{%
V(R)}{V(s)}sds}{V(R)}\right) ^{2} \\
&\leq & \frac{V(R)}{V(r)}\left( 1+C\frac{V(r)}{r^{2}}\frac{\int_{r}^{\infty
}\left( \frac{R}{s}\right) ^{2+\gamma _{1}}sds}{V(R)}\right) ^{2} \\
&\leq & C\frac{V(R)}{V(r)}\left( 1+\frac{V(r)}{V(R)}\left( \frac{R}{r}\right)
^{2+\gamma _{1}}\right) ^{2} \\
&\leq & C\left( \frac{R}{r}\right) ^{2+\gamma _{1}}+C\left( \frac{R}{r}%
\right) ^{2+2\gamma _{1}+\gamma _{2}}\leq C\left( \frac{R}{r}\right)
^{2+2\gamma _{1}+\gamma _{2}}
\end{eqnarray*}%
\end{small}
which proves (\ref{Vtil2}).
\end{proof}

\begin{lemma}
\label{lemma Psi} Assume that the manifold $M$ is either subcritical or
regular. Then there exist $C,\lambda _{0}>0$ depending on $K$ and $A$, such
that, for all $\lambda \in \left( 0,\lambda _{0}\right) $, 
\begin{equation}
\sup_{\partial A}\Psi _{\lambda }^{K^{c}}\leq \frac{C}{\lambda ^{2}%
\widetilde{V}(\frac{1}{\sqrt{\lambda }})}.  \label{Psi}
\end{equation}
\end{lemma}

\begin{proof}
Set $\Omega =K^{c}$. By \cite[(4.48)]{G-SC hitting} we have, for all $x\in
\partial A$ and $t\geq t_{0}>0$ 
\[
\partial _{t}\mathbb{P}_{x}\left( \tau _{\Omega }\leq t\right) \leq \frac{C}{%
\widetilde{V}\left( \sqrt{t}\right) }.
\]%
Taking $\lambda _{0}\leq \frac{1}{t_{0}}$, we obtain by (\ref{PsiOm}) for
any $0<\lambda <\lambda _{0}$%
\begin{eqnarray}
\Psi _{\lambda }^{\Omega }\left( x\right) &=&\int_{0}^{t_{0}}te^{-\lambda
t}\partial _{t}\mathbb{P}_{x}(\tau _{\Omega }\leq t)dt+\int_{t_{0}}^{\infty
}te^{-\lambda t}\partial _{t}\mathbb{P}_{x}(\tau _{\Omega }\leq t)dt  \nonumber
\\
&\leq &t_{0}+C\int_{t_{0}}^{\infty }e^{-\lambda t}\frac{t}{\widetilde{V}%
\left( \sqrt{t}\right) }dt.  \label{problem1}
\end{eqnarray}%
Observe that if $M$ is regular then by (\ref{Vtil1}) we have for all $r\geq
1 $%
\begin{equation}
\widetilde{V}(r)\leq Cr^{4}.  \label{problem3}
\end{equation}
If $M$ is subcritical, we have 
\begin{equation}
\widetilde{V}(r)=V\left( r\right) h^{2}\left( r\right) \simeq V\left(
r\right) \left( \frac{r^{2}}{V\left( r\right) }\right) ^{2}=\frac{r^{4}}{%
V\left( r\right) },  \label{r4/}
\end{equation}
which again implies (\ref{problem3}).

In particular, it follows from (\ref{problem3}) for $r=\frac{1}{\sqrt{%
\lambda }}$ that 
\begin{equation}
\frac{1}{\lambda ^{2}\widetilde{V}(\frac{1}{\sqrt{\lambda }})}\geq c
\label{t0}
\end{equation}%
so that the constant term $t_{0}$ in (\ref{problem1}) can be estimated from
above as follows: 
\begin{equation}
\,t_{0}\lesssim \frac{1}{\lambda ^{2}\widetilde{V}(\frac{1}{\sqrt{\lambda }})%
}.  \label{in1}
\end{equation}%
Next, let us decompose the integral in (\ref{problem1}) into two intervals: $%
[t_{0},1/\lambda ]$ and $[1/\lambda ,\infty )$. For the second interval, we
obtain 
\begin{equation}
\int_{1/\lambda }^{\infty }e^{-\lambda t}\frac{t}{\widetilde{V}\left( \sqrt{t%
}\right) }dt\leq \frac{1}{\widetilde{V}(\frac{1}{\sqrt{\lambda }})}%
\int_{1/\lambda }^{\infty }te^{-\lambda t}dt\simeq \frac{1}{\lambda ^{2}%
\widetilde{V}(\frac{1}{\sqrt{\lambda }})}.  \label{in2}
\end{equation}%
Let us estimate the first integral. If $M$ is regular then, by using (\ref%
{Vtil2}), we obtain 
\begin{eqnarray}
\int_{t_{0}}^{1/\lambda }e^{-\lambda t}\frac{t}{\widetilde{V}\left( \sqrt{t}%
\right) }dt& \leq &\frac{1}{\widetilde{V}(\frac{1}{\sqrt{\lambda }})}%
\int_{t_{0}}^{1/\lambda }\frac{\widetilde{V}\left( \frac{1}{\sqrt{\lambda }}%
\right) }{\widetilde{V}\left( \sqrt{t}\right) }tdt  \nonumber \\
& \lesssim &\frac{1}{\lambda ^{2}\widetilde{V}(\frac{1}{\sqrt{\lambda }})}%
\int_{0}^{1/\lambda }\left( \frac{1}{\lambda t}\right) ^{1+\gamma
_{1}+\gamma _{2}/2}\left( \lambda t\right) d\left( \lambda t\right)   \nonumber
\\
& =&\frac{1}{\lambda ^{2}\widetilde{V}(\frac{1}{\sqrt{\lambda }})}%
\int_{0}^{1}s^{-\left( \gamma _{1}+\gamma _{2}/2\right) }ds  \nonumber \\
& \lesssim &\frac{1}{\lambda ^{2}\widetilde{V}(\frac{1}{\sqrt{\lambda }})},
\label{in3r}
\end{eqnarray}%
because $\gamma _{1}+\gamma _{2}/2<1.$ 

If $M$ is subcritical, then we use the fact that $\left( VD\right) $ implies
the \emph{reverse volume doubling} (see, for instance, \cite[Lemma 5.2.8]{SC
LNS}), that is, for some $\theta >0$ and for all $R\geq r\geq 1,$%
\begin{equation}
\frac{V(r)}{V\left( R\right) }\lesssim \left( \frac{r}{R}\right) ^{\theta }.
\label{reverse}
\end{equation}%
Then we obtain by (\ref{r4/}) and (\ref{reverse}) that%
\begin{eqnarray}
\int_{t_{0}}^{1/\lambda }e^{-\lambda t}\frac{t}{\widetilde{V}\left( \sqrt{t}%
\right) }dt& \leq &\int_{t_{0}}^{1/\lambda }\frac{t}{\widetilde{V}\left( 
\sqrt{t}\right) }dt  \nonumber \\
& \simeq & \int_{t_{0}}^{1/\lambda }\frac{t}{t^{2}/V\left( \sqrt{t}\right) }dt
\nonumber \\
& =&V(\frac{1}{\sqrt{\lambda }})\int_{t_{0}}^{\frac{1}{\lambda }}\frac{V(%
\sqrt{t})}{V(\frac{1}{\sqrt{\lambda }})}\frac{1}{t}dt  \nonumber \\
& \lesssim &V(\frac{1}{\sqrt{\lambda }})\int_{0}^{\frac{1}{\lambda }}(\lambda
t)^{\theta /2}\frac{1}{\lambda t}d\left( \lambda t\right)   \nonumber \\
& \lesssim &V(\frac{1}{\sqrt{\lambda }})\simeq \frac{1}{\lambda ^{2}%
\widetilde{V}(\frac{1}{\sqrt{\lambda }})}.  \label{in3s}
\end{eqnarray}%
Combining (\ref{in1}), (\ref{in2}), (\ref{in3r}) and (\ref{in3s}), we obtain
(\ref{Psi}).
\end{proof}

\begin{remark}
If $M$ is critical, that is, $V\left( r\right) \simeq r^{2},$ then Lemma \ref%
{lemma Psi} implies that 
\[
\sup_{\partial A}\Psi _{\lambda }^{K^{c}}\leq \frac{C}{\lambda \log ^{2}%
\frac{1}{\lambda }},
\]%
which coincides with \cite[(3.39)]{GIS}. If $M$ is subcritical, then Lemma %
\ref{lemma Psi} implies that 
\[
\sup_{\partial A}\Psi _{\lambda }^{K^{c}}\leq CV(\frac{1}{\sqrt{\lambda }}).
\]%
Using further (\ref{subcritical upper}), we obtain%
\[
\sup_{\partial A}\Psi _{\lambda }^{K^{c}}\leq C\left( \frac{1}{\lambda }%
\right) ^{1-\delta /2},
\]%
which is a significant improvement of \cite[(3.39)]{GIS}.
\end{remark}


\subsection{Proof of Theorem \protect\ref{main}}

\label{SecProofTmain}Here we prove Theorem \ref{main}. Recall that we
consider a manifold with nice ends $M=\#_{i\in I}M_{i}$ where each end is parabolic 
satisfying (RCA) (see
Figure \ref{figure: connectedsum}). Besides, we assume that each end satisfies 
either regular or subcritical, and 
also, (DOE) if there exists at least one non-subcritical regular end (see Definition \ref{DOE}). 
Our aim is to prove the estimate (\ref{ptoo}) that is, 
\begin{equation}
p(t,o,o)\simeq \frac{1}{V_m(\sqrt{t})},  \label{ptoo2}
\end{equation}%
where $V_{m}(r)$ is the largest volume function at scale $r$. 

 As before, let $K$ be the central part of $M$. Let $A$ be a
connected, precompact open subset of $M$ with smooth boundary and such that $%
K_{\varepsilon }\subset A$ for some large enough $\varepsilon >0$. Set 
\[
\partial A_{i}:=(\partial A)\cap E_{i},~~i=1,\ldots ,k.
\]%

First we prove the following heat kernel upper bound.
\begin{proposition}
\label{prop min min}
Let $M=\#_{i \in I} M_i$ be a manifold with parabolic ends, where each $M_i$ 
is either regular (Definition \ref{DefRegular}) or subcritical (Definition \ref{DefCOE}).
Then for any $t\gg 1$
\begin{equation}
p(t,o,o)\lesssim \frac{\min_{i}h_{i}^{2}(\sqrt{t})}{\min_{i}\widetilde{V}%
_{i}(\sqrt{t})}.  \label{refine}
\end{equation}%
\end{proposition}

We use here the following two Lemmas from \cite{GIS} that were proved for
arbitrary manifolds with ends.

\begin{lemma}
({\cite[Lemma 4.1]{GIS})} There exists a constant $C=C(A,K)>0$ such that,
for all $\lambda >0$, 
\begin{equation}
(\sup_{\partial K}\gamma _{\lambda })\sum_{i=1}^{k}\inf_{\partial A_{i}}\Phi
_{\lambda }^{E_{i}}\leq C.  \label{GIS Lemma4.1}
\end{equation}
\end{lemma}

The estimate (\ref{GIS Lemma4.1}) combined with the estimates in (\ref{p
gamma}) and (\ref{Phi}) implies that 
\[
p(t,o,o)\leq \frac{1}{t}\min_{i}h_{i}(\sqrt{t}).
\]%
As was pointed out in \cite[Remark 3.2]{GIS}, this estimate gives the
optimal upper bound of $p(t,o,o)$ when all $M_{i}$ are subcritical. However,
if there exists at least one critical end, this estimate yields 
\[
p(t,o,o)\leq C\frac{\log t}{t},
\]%
instead of the desired bound 
\[
p(t,o,o)\leq \frac{C}{t}.
\]%
In order to obtain an optimal bound of $p(t,o,o)$ on parabolic manifolds with
non-subcritical ends, we will use the following result.

\begin{lemma}
({\cite[Lemma 4.2]{GIS})} There exists a constant $C>0$ depending on $A,K$
such that for any $\lambda >0$ 
\begin{equation}
(\sup_{\partial K}\dot{\gamma}_{\lambda })\sum_{i=1}^{k}\inf_{\partial
A_{i}}\Phi _{\lambda }^{E_{i}}\leq C+(\sup_{\partial K}\gamma _{\lambda
})\left( C+\sum_{i=1}^{k}\sup_{\partial A_{i}}\Psi _{\lambda
}^{E_{i}}\right) .  \label{GIS Lemma4.2}
\end{equation}
\end{lemma}

\begin{proof}[Proof of Proposition \ref{prop min min}]
Substituting (\ref{Phi}) into (\ref{GIS
Lemma4.1}), we obtain 
\begin{equation}
\sup_{\partial K}\gamma _{\lambda }\lesssim \frac{1}{\sum_{i=1}^{k}\inf_{%
\partial A_{i}}\Phi _{\lambda }^{E_{i}}}\lesssim \frac{1}{\sum_{i=1}^{k}%
\frac{1}{h_{i}(\frac{1}{\sqrt{\lambda }})}}\lesssim \min_{i\in I}h_{i}(\frac{%
1}{\sqrt{\lambda }}):=w\left( \lambda \right) .  \label{gamma h}
\end{equation}%
Applying  (\ref{GIS Lemma4.2}) with (\ref{gamma h}) and (\ref{Psi}), we
obtain, for small $\lambda >0$, 
\[
\sup_{\partial K}\dot{\gamma}_{\lambda }\lesssim w\left( \lambda \right)
\left( 1+w\left( \lambda \right) \left( 1+\sum_{i=1}^{k}\sup_{\partial
A_{i}}\Psi _{\lambda }^{E_{i}}\right) \right) \lesssim \frac{w^{2}\left(
\lambda \right) }{\lambda ^{2}\min_{i}\widetilde{V}_{i}(\frac{1}{\sqrt{%
\lambda }})},
\]%
where we have also used that $w\geq 1$ and, by (\ref{problem3}),%
\[
\lambda ^{2}\min_{i}\widetilde{V}_{i}(\frac{1}{\sqrt{\lambda }})\lesssim 1.
\]%
Hence, we have proved that  
\[
\sup_{\partial K}\dot{\gamma}_{\lambda }\lesssim \frac{\min_{i}h_{i}^{2}(%
\frac{1}{\sqrt{\lambda }})}{\lambda ^{2}\min_{i}\widetilde{V}_{i}(\frac{1}{%
\sqrt{\lambda }})}.
\]%
Taking here $\lambda =t^{-1}$ and substituting this estimate of $\dot{\gamma}%
_{\lambda }$ into (\ref{p gamma dot}), we conclude Proposition \ref{prop min min}.
\end{proof}

\begin{proof}[Proof of Theorem \ref{main} (estimate (\ref{ptoo2}))] 
First, if all ends are subcritical, then we have by definition
\[
h_i(r) \simeq \frac{r^2}{V_i(r)}, \quad \widetilde{V}_i(r) \simeq r^2.
\]
Then the estimate in (\ref{refine}) implies that 
\begin{equation}
p(t, o, o) \lesssim \frac{ \min_i \frac{t}{V_i(\sqrt{t})} }{ t} =\frac{1}{V_m(\sqrt{t})}
\simeq \frac{1}{V(o,\sqrt{t})},
\label{subcritical heat upper}
\end{equation}
where $V(o,r)$ is the measure of the geodesic ball in $M$ with radius $r$ 
centered at $o $.

If there exists at least one non-subcritical regular end, then 
by the assumption of (DOE), we have
\[
\min_{i}h_{i}^{2}(r)\simeq h_{l }^{2}(r),~~\min_{i}\widetilde{V}_{i}(r)\simeq 
\widetilde{V}_{l }(r).
\]%
Substituting this into the estimate in (\ref{refine}), it follows that for all $t\gg 1$ 
\begin{equation}
p(t,o,o)\lesssim \frac{h_{l }^{2}\left( \sqrt{t}\right) }{\widetilde{V}%
_{l }\left( \sqrt{t}\right) }=\frac{1}{V_{l }(\sqrt{t})}.  \label{max}
\end{equation}%
Because $V\left( o,r\right) \simeq V_m(r) \simeq V_{l }\left( r\right) $ 
(see Lemma \ref{Lem2}), we obtain, for all $%
t\gg 1$%
\[
p(t,o,o)\lesssim \frac{1}{V\left( o,\sqrt{t}\right) },
\]%
which gives the same upper estimate as in (\ref{subcritical heat upper}).

For a bounded range of $t$, this estimate is trivially satisfied, so that it
holds for all $t>0$. Since $V_{m}\left( r\right) $ is doubling, the volume
function $V\left( o,r\right) $ on $M$ is also doubling. By \cite%
{Coulhon-Grigoryan}, we conclude that $p\left( t,o,o\right) $ satisfies also
a matching lower bound%
\[
p(t,o,o)\gtrsim \frac{1}{V\left( o,\sqrt{t}\right) }\simeq \frac{1}{%
V_{m}\left( \sqrt{t}\right) }.
\]
Hence, we obtain the estimate in (\ref{ptoo2}), which concludes 
Theorem \ref{main}.
\end{proof}


\begin{remark}
If all ends $M_{i}$ are either critical (i.e., $V(r)\simeq r^{2}$) or
subcritical, then (\ref{ptoo}) was already obtained in our previous
paper \cite[Theorem 2.1]{GIS}. When all ends are subcritical, $%
m$ may depend on $r$ (see Theorem \ref{GIS}).

If {\rm{(DOE)}} is not satisfied then the upper estimate in (\ref{max}) might not
be optimal. We expect to obtain such an example in the case $M=M_{1}\#M_{2}$
with $V_{i}(r)=r^{2}\varphi _{i}(r)$, where each $\varphi _{i}$ is a slowly
varying function and $V_{1},V_{2}$ are not properly ordered in any of the
above sense. See \cite{GISfluc} for an example of volume functions $V_1, V_2$ behaving 
that way.
\end{remark}

\section{Estimate of the Poincar\'{e} constant}
\setcounter{equation}{0}
\label{SecPI}

\subsection{Remarks on the conditions {\rm{(COE)}}, {\rm{(DOE)}} and regularity}

\label{V Vh Vtil}In the next two lemmas we collect some already known
properties.

\begin{lemma}
\label{lemma V Vh Vtil}\label{Lem1} 
Let $V$ be regular with parameters 
$\gamma_1, \gamma_2>0$
satisfying $2\gamma_1+\gamma_2 <2$. 
Then for all $r\geq 1$ we have 
\begin{eqnarray}
r^{2-\gamma _{2}}&\lesssim & V(r)\lesssim r^{2+\gamma _{1}},  \label{V-1} \\
r^{2}&\lesssim & V(r)h(r)\lesssim r^{2+\gamma _{1}},  \label{V-2} \\
V(r)&\lesssim & \widetilde{V}(r)\lesssim r^{2+\gamma _{1}+\gamma _{2}}.
\label{V-3}
\end{eqnarray}
\end{lemma}

\begin{proof}
The estimates (\ref{V-1}) follow immediately from (\ref{regular}). The lower
bound in (\ref{V-2}) is equivalent to (\ref{hi>}).  The upper bound in (\ref%
{V-2}) is equivalent to (\ref{Vh}). 

The lower bound in (\ref{V-3}) is trivial because $h\geq 1.$ The upper bound
in (\ref{V-3}) coincides with (\ref{Vtil1}).
\end{proof}


\begin{lemma}
\label{Lem2}Let $M=\#_{i \in I} M_i$ be a manifold with nice ends. Assume that $M$ 
satisfies $(COE)$ with parameters $\varepsilon ,\delta
,\gamma _{1},\gamma _{2}$ (see Definition \emph{\ref{DefCOE}}). 
\begin{itemize}
\item In general, we have:
 \begin{itemize}
\item[(a)]  If $i\in I_{super}$, then $M_{i}$ is non-parabolic and $%
h_{i}(r)\simeq 1.$

\item[(b)] If $i\in I_{super}$, $j\in I_{middle}$ and $k\in I_{sub}$%
, then, for all $r\gg 1$,%
\begin{equation}
V_{i}(r)\gtrsim V_{j}(r)\gtrsim V_{k}(r)~\mbox{ and }V_{i}(r)\gtrsim
V_{j}(r)h_{j}(r)\gtrsim V_{k}(r)h_{k}(r)\simeq r^{2}  \label{(ii)-1}
\end{equation}%
and 
\begin{equation}
\widetilde{V}_{j}(r)\lesssim \widetilde{V}_{k}(r).  \label{(ii)-2}
\end{equation}%
\end{itemize}

\item If all ends are parabolic and $I_{middle}\neq \emptyset $ 
then the following properties hold:

\begin{itemize}
\item[(c)]
$I_{super}=\emptyset$ and 
$M$ admits $(DOE)$, introduced in Definition \ref{DOE}, that is, 
there exists a dominating volume function 
$V_{l}$ with $l \in I_{middle}$, that is a volume function such that for all $i \in I$
\begin{equation}
V_{l} \gtrsim V_i, \quad \widetilde{V}_{l} \lesssim \widetilde{V}_i.
\label{dominate}
\end{equation}

\item[(d)]
For all $r\gg 1$,
\begin{eqnarray*}
V_{l} (r) \simeq V_m(r) ~\mbox{and}& ~ \min_i h_i (r) \simeq h_{l}  (r) \simeq h_m(r),
\end{eqnarray*}
where $m=m(r)$ is the  index of the largest end
defined in (\ref{largest end}).

\item[(e)]
 For all $r\gg 1$,
\begin{eqnarray*}
\max_{j\neq l} \left\{ V_j(r) h_j(r)  \right\} \simeq &V_n(r)h_n(r),
\end{eqnarray*}
where $n=n(r)$ is the index of the second largest end
defined in (\ref{second largest end}).
\end{itemize}
\end{itemize}

\end{lemma}

\begin{proof}
$\left( a\right) $ For any $i\in I_{super}$ we have by definition $%
V_{i}\left( r\right) \gtrsim r^{2+\varepsilon }$ and, hence, 
\[
1\leq h_{i}(r)\lesssim 1+\int_{1}^{r}\frac{1}{s^{2+\varepsilon }}sds\lesssim
1.
\]
This means non-parabolicity of $M_i$ because of (\ref{LY}).

$\left( b\right) $ By (\ref{gade}) $\varepsilon >\gamma _{1}$ and $%
\gamma _{2}<\delta $, we obtain, by using Definition \ref{DefCOE} that%
\[
V_{i}\left( r\right) \gtrsim r^{2+\varepsilon }\gtrsim r^{2+\gamma
_{1}}\gtrsim V_{j}\left( r\right) \gtrsim r^{2-\gamma _{2}}\gtrsim
r^{2-\delta }\gtrsim V_{k}\left( r\right) .
\]
By (\ref{V-2}), subcriticality of $V_{k}$ and $\varepsilon >\gamma _{1}$ we
obtain%
\[
V_{k}(r)h_{k}(r)\simeq r^{2}\lesssim V_{j}(r)h_{j}(r)\lesssim r^{2+\gamma
_{1}}\lesssim r^{2+\varepsilon }\lesssim V_{i}\left( r\right) ,
\]%
which proved (\ref{(ii)-1}).

Since $V_{k}$ is subcritical, we have%
\[
\widetilde{V}_{k}(r)\simeq \frac{r^{4}}{V_{k}(r)}\geq cr^{2+\delta }.
\]%
Because $\gamma _{1}+\gamma _{2}<\delta $, we obtain by (\ref{V-3}) 
\[
\widetilde{V}_{j}(r)\lesssim r^{2+\gamma _{1}+\gamma _{2}}\lesssim
r^{2+\delta }\lesssim \widetilde{V}_{k}(r).
\]%

Now we assume that all ends are parabolic and $I_{middle}\neq
\emptyset $. 

$\left( c \right)$ 
 (a) implies $I_{super}=\emptyset$ immediately. 
By the definition of $I_{middle}$ and the estimates in (\ref{(ii)-1}), 
there exists an end $l \in I_{middle}$ such that
 for all $i \in I$ 
\begin{equation}
V_{l} \gtrsim V_i.
\label{*i}
\end{equation}
By the estimates in (\ref{(ii)-2}), we obtain
\[
\widetilde{V}_{l } \lesssim \widetilde{V}_i,
\]
which concludes (DOE).

(d) $V_{l} \gtrsim V_m$ and $V_m(r) \ge V_{l} (r)$ imply
$V_{l} (r) \simeq V_m(r)$ immediately.
By the estimates in (\ref{*i}),  we obtain
\[
\min_{i\in I} h_i \simeq h_{l}.
\]
Next, $V_{l} \gtrsim V_m$ implies that 
$h_{l} \lesssim h_m$ and, by the assumption of (COE) 
\[
V_{l} h_{l} \gtrsim V_i h_i.
\]
Since $V_{l} (r) \simeq V_m(r)$ at $r$, we obtain 
\[
h_{l} (r) \gtrsim h_m(r),
\]
which concludes $h_{l} (r) \simeq h_m(r)$.

(e) If $|I_{middle}|=1$, then for all $r\gg 1$, $l=m(r)$ and the second largest end is 
subcritical. Hence we can prove (e) easily (see also Lemma \ref{Lem2}).

If $|I_{middle}|\ge 2$, then for all $r\gg 1$, $m$ and $n$ are in $I_{middle}$. 
Assume first that $m=m(r)=l $. Let $j\in I_{middle}$ be the index 
so that $V_j(r)h_j(r) =\max_{i \in I \backslash \{ l \} } \left\{ V_i (r) h_i(r) \right\}$.
By the condition of (COE), Either $V_n\gtrsim V_j$ or $V_j \gtrsim V_n$ holds.
If $V_n\gtrsim V_j$, then $V_nh_n \gtrsim V_j h_j$, which implies (e).
If $V_j \gtrsim V_n$, we note that $V_n(r) \ge V_j(r)$ by $j\neq m=l $. 
Then same argument as in (a) concludes $V_j(r) \simeq V_n(r)$ and $h_j(r) \simeq h_n(r)$
which concludes (e).

Next, assume that $m \neq l $. In this case, we have
\begin{eqnarray*}
V_{l}  \gtrsim  V_m &~\mbox{and}~& V_{l} \gtrsim V_n,\\
V_m(r) \ge & \!\!\!\!\! V_n(r)&\!\!\!\!\! \ge V_{l} (r).
\end{eqnarray*}
Then we obtain $V_{l} (r) \simeq V_m(r) \simeq V_n(r)$ and, by the same argument as in (d),
we obtain
\[
h_{l} (r) \simeq h_m(r) \simeq h_n(r),
\]
whence we conclude (e).
\end{proof}

\subsection{Lower bound on Dirichlet heat kernel}

\label{SecLBDirichlet}Recall that $M$ is a manifold with nice ends $M_{1},\ldots
,M_{k}$. Moreover, assume (RCA) on
each parabolic end. First, we prove the following technical lemmas.

\begin{lemma}
\label{technical} Let $M$ be as above. Assume also that for all $t>0$ 
\begin{equation}
p(t, o, o) \simeq \frac{ \min_i h_i^2(\sqrt{t})}{\min_i \widetilde{V}_i(%
\sqrt{t})}.  \label{min min}
\end{equation}
Then we have for all $r\gg 1$ 
\begin{eqnarray}
\int_1^{r^2} p(s, o,o)ds \simeq &\min_i h_i(r).  \label{property1}
\end{eqnarray}
\end{lemma}

\begin{proof}
If $M$ is non-parabolic, although the estimates in (\ref{min min}) holds
(see Theorem \ref{non-parabolic}), we can prove the estimate (\ref{property1}%
) without using (\ref{min min}). Indeed, by the definition of the
non-parabolicity we have for $r \gg 1$ 
\[
\int_1^{r^2} p(s, o,o)ds \simeq 1.
\]
Because $h_i \simeq 1$ for each non-parabolic end, we have 
\[
\min_i h_i \simeq 1.
\]
Then we conclude (\ref{property1}).

Next, suppose that $M$ is parabolic. By the assumption (\ref{min min}), 
\[
\int_1^{r^2} p(s, o,o)ds \simeq \int_1^{r} \frac{ \min_i h_i^2(s)}{\min_i 
\widetilde{V}_i (s)} sds.
\]
Observe that 
\[
h_i(r) = \frac{1}{\int_r^\infty \frac{sds}{\widetilde{V}_i(s)} },
\]
which is obtained by integration of the identity 
\[
\left( \frac{1}{h(r)} \right)^\prime =-\frac{h^\prime(r)}{h^2(r)} =-\frac{%
r/V(r)}{h^2(r)}=-\frac{r}{\widetilde{V}_i(r)}
\]
and 
\[
\frac{1}{h(r)} \rightarrow 0 ~\mbox{as}~ r\rightarrow \infty
\]
by the parabolicity. Then we have 
\begin{eqnarray*}
\min_i h_i(r)& =&\frac{1}{\max_i \int_r^\infty \frac{sds}{\widetilde{V}_i(s)} }
\simeq \frac{1}{\sum_i \int_r^\infty \frac{sds}{\widetilde{V}_i(s)} } \\
&\simeq & \frac{1}{ \int_r^\infty \frac{sds}{\min_i \widetilde{V}_i(s)} } = 
\frac{1}{\int_r^\infty \frac{sds}{W(s)} }=:h(r),
\end{eqnarray*}
where 
\begin{eqnarray*}
W(r)=\min_i \widetilde{V}_i (r).
\end{eqnarray*}

Because
\[
-\frac{h^\prime(r)}{h^2(r)}= \left( \frac{1}{h(r)}\right)^\prime=-\frac{r}{%
W(r)}
\]
and $h$ is a doubling function, we conclude that for all $r \ge 2$ 
\[
\int_1^{r} \frac{ \min_i h_i^2 (s)}{\min_i \widetilde{V}_i (s)} sds \simeq
\int_1^r \frac{h^2(s)}{W(s)}sds =\int_1^r h^\prime (s)ds =h(r)-h(1) \simeq
h(r).
\]
\end{proof}

For $i=1, \ldots , k$, set 
\begin{equation}
A_{i}(r)=\left( B(o,r)\backslash B(o,r/2)\right) \cap E_{i}.
\label{annulus}
\end{equation}
For $\kappa >0$, $\kappa B$ means $B(o,\kappa r)$ . Recall that $p^D_{\kappa
B}(t,x,y)$ is the Dirichlet heat kernel on $\kappa B$.

The following estimates of $p^D_{\kappa B}(t,x,y)$ has a key role for the
estimate of the Poincar\'{e} constant:

\begin{lemma}
\label{Dirichlet} Assume that a connected sum $M=M_{1}\#\ldots \#M_{k}$
satisfies all the hypothesis of Lemma \ref{technical}. Then there exists $%
\kappa \geq 1$ such that for all $r\gg 1$, for any $x\in A_{i}(r)$, $y\in
B(o,r)\cap E_{j}$, 
\begin{small}
\begin{equation}
p_{\kappa B}^{D}(r^{2},x,y)\simeq \frac{1}{V_{i}(r)h_{j}(r)}\Bigg[\delta
_{ij}h_{j}(|y|)+\frac{\min_{\eta } h_{\eta} (r)}{h_{i}(r)}\left( \int_{|y|}^{r}\frac{sds%
}{V_{j}(s)}+\frac{r^{2}h_{j}(|y|)}{V_{j}(r)h_{j}(r)}\right) \Bigg].
\label{estimate Dirichlet}
\end{equation}%
\end{small}
Moreover, this implies that
\begin{equation}
\inf_{x\in A_{i}(r),y\in B}p_{\kappa B}^{D}(r^{2},x,y)\gtrsim  \frac{r^{2} \min_{\eta} h_{\eta} (r)}
{V_i(r)h_i(r) \max_{j \neq i} \left\{ V_j(r) h_j(r)\right\}  }.
\label{lower Dirichlet}
\end{equation}
%
%
%
%
%
%
%
%
%
%
%
%
%
%
%
%
%
%
%
%
\end{lemma}


\begin{proof}
We prove the estimate in (\ref{estimate Dirichlet}) only from below (upper
bound follows from the trivial bound $p_{\kappa B}^{D}(r^{2},x,y)\leq
p(r^{2},x,y)$ and the estimate in (\ref{step1-2})).

Recall that $\tau_{\kappa B}$ is the first exist time from $\kappa B$. By
the definition of the Dirichlet heat kernel and the strong Markov property
of the Brownian motion on $M$ (see also \cite{G-SC ends}), we obtain 
\begin{eqnarray*}
p_{\kappa B}^{D}(r^{2},x,y)=& p(r^{2},x,y)-\mathbb{E}_{y}(\mathbf{1}_{\{\tau
_{\kappa B}\leq r^{2}\}}p(r^{2}-\tau _{\kappa B}, X_{\tau _{\kappa B}} ,x)) ,
\end{eqnarray*}%
where $\kappa \geq 1$ will be chosen later. By the structure of the
connected sum, we observe that 
\[
\{ \tau_{\kappa B} \leq r^2 \} = \cup_{\xi =1}^k \{ \tau_{\kappa B} \leq r^2, ~
X_{\tau_{\kappa B}} \in E_{\xi } \} \subset \cup_{\xi =1}^k \left\{ \tau_{
\big( (\kappa B)^c\cap E_{\xi } \big)^c} \leq r^2 \right\}
\]
(see Figure \ref{Figure: hitting}). 
\begin{figure}[tbph]
\begin{center}
\scalebox{1.1}{
\includegraphics{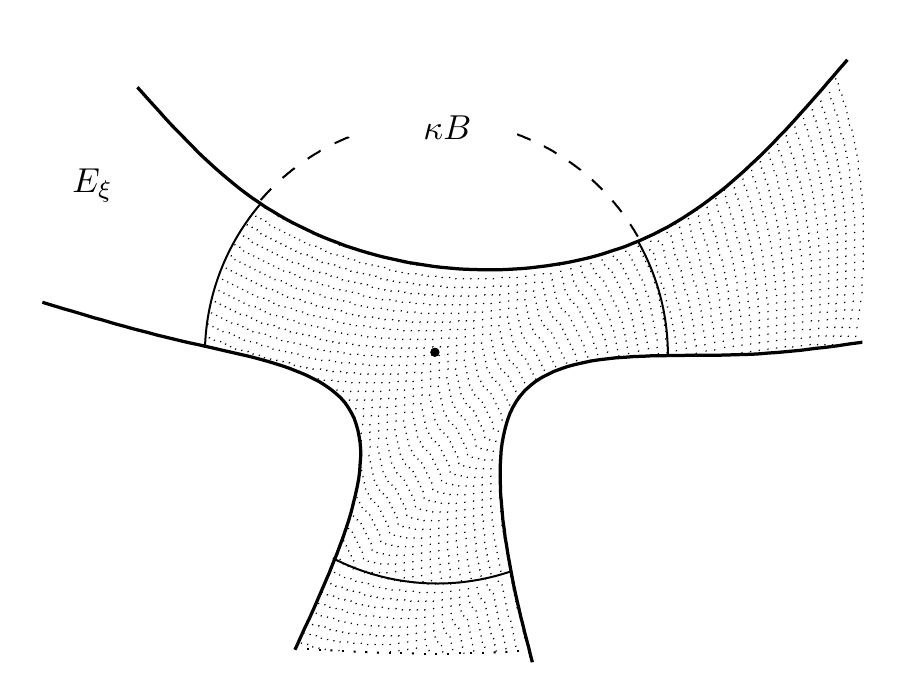} 
}
\end{center}
\par
\label{Figure: hitting}
\caption{$\Big( (\protect\kappa B)^c \cap E_{\xi } \Big)^c$ (shadow area).}
\end{figure}

We note that 
\[
\partial \Big( (\kappa B)^c \cap E_{\xi } \Big)^c =(\partial \kappa B) \cap E_{\xi }.
\]

Then we obtain 
\begin{small}
\begin{eqnarray}
p_{\kappa B}^{D}(r^{2},x,y)&=& p(r^2, x,y)-\sum_{\xi =1}^k \mathbb{E}_{y}\left( 
\mathbf{1}_{\{ \tau_{\kappa B} \leq r^2 , ~X_{\tau_{\kappa B} \in E_{\xi }} \}}
p(r^2- \tau_{\kappa B} , X_{\tau_{\kappa B}} , x) \right) \nonumber\\
&\geq & p(r^{2},x,y)-\sum_{\xi =1}^{k}\mathbb{P}_{y}\left( \tau _{\big((\kappa B)^{c}\cap
E_{\xi }\big)^{c}}\leq r^{2} \right) 
\sup_{\stackrel{0<s\leq r^{2}}{z \in \left(\partial \kappa B\right) \cap E_{\xi }}}p(s, z ,x). \nonumber\\
\label{Dirichlet decomp}
\end{eqnarray}%
\end{small}
We estimate $p(r^2 ,x,y)$, $\mathbb{P}_{y} 
\Big( \tau _{\big( (\kappa B )^{c}\cap E_{\xi } \big)^{c}}\leq r^2 \Big)$ 
and $\sup_{\stackrel{0<s\leq r^2}{z \in \left( \partial \kappa
B\right) \cap E_{\xi }}}p(s, z, x)$ separately.

\textbf{Step1: Estimate of }$p(r^2 ,x,y)$

We will prove that for any $x \in A_i(r)$, $y \in B(o,r)\cap E_j$
\begin{small}
\begin{equation}
p(r^{2},x,y)\simeq \frac{1}{V_{i}(r)h_{j}(r)}\Bigg[\delta _{ij}h_{j}(|y|)+%
\frac{\min_{\eta} h_{\eta} (r)}{h_{i}(r)}\left( \int_{|y|}^{r}\frac{sds}{V_{j}(s)}+%
\frac{r^{2}h_{j}(|y|)}{V_{j}(r)h_{j}(r)}\right) \Bigg],  \label{step1-2}
\end{equation}%
\end{small}
which is nothing but the same estimate as in (\ref{estimate Dirichlet}) for the
full heat kernel $p(r^{2},x,y)$ instead of 
$p^D_{\kappa B}(r^2, x, y)$. Because $p^{D}_{\kappa B}(r^2, x, y) \leq 
p(r^2, x, y)$ is always true, this estimate gives the upper bound in
 (\ref{estimate Dirichlet}).

To prove (\ref{step1-2}), we use Theorem \ref{off-diagonal} and the location of 
$x,y$ to obtain 
\begin{small}
\begin{eqnarray}
p(r^{2},x,y)&\simeq & \delta _{ij}\frac{1}{V_{i}(r)}\frac{h_{j}(|y|)}{h_{j}(r)%
}+p(r^{2},o,o)\frac{r^{2}}{V_{i}(r)h_{i}(r)}\frac{1}{h_{j}(r)}\int_{|y|}^{r}%
\frac{sds}{V_{j}(s)}  \nonumber \\
&&\!\!\!\!\!\!\!\!\!\!\!\!\!\!\!\!\!\!\!\!\!\!\!\! +\int_{1}^{r^{2}}p(s,o,o)ds\left( \frac{1}{V_{i}(r)h_{i}(r)}\frac{1}{%
h_{j}(r)}\int_{|y|}^{r}\frac{sds}{V_{j}(s)}+\frac{r^{2}}{V_{i}(r)h_{i}(r)}%
\frac{h_{j}(|y|)}{V_{j}(r)h_{j}^{2}(r)}\right) .  \nonumber \\
\label{step1}
\end{eqnarray}%
\end{small}
Because $p(t,o,o)$ is monotone decreasing, it is easy to see that 
\begin{equation}
r^{2}p(r^{2},o,o)\leq C\int_{1}^{r^{2}}p(s,o,o)ds,  \label{trivial}
\end{equation}%
which makes the third term of the right hand side in (\ref{step1}) dominate
the second term. By using Lemma \ref{technical}, we obtain (\ref{step1-2}).

\textbf{Step2: Estimate of $\mathbb{P}_{y} 
\Big( \tau _{\big( (\kappa B )^{c}\cap E_{\xi } \big)^{c}}\leq r^2 \Big)$}

We will prove that for any $y \in B(o, r) \cap E_j$
\begin{equation}
\mathbb{P}_{y}\left( \tau _{\big( (\kappa B)^{c}\cap E_{\xi } \big)^c}\leq r^{2} \right) 
\asymp C \Big[%
\delta _{\xi j}\frac{h_{j}(|y|)}{h_{j}(r)}+\frac{\min_{\eta} h_{\eta} (r)}{h_{\xi}(r)}\frac{%
1}{h_{j}(r)}\int_{|y|}^{r}\frac{sds}{V_{j}(s)}\Big]e^{-b\kappa ^{2}}.
\label{step2}
\end{equation}

Since $y\in B(o,r)\cap E_{j}$, applying Theorem 3.5 in \cite{G-SC hitting},
we obtain 
\[
\mathbb{P}_{y}(\tau _{((\kappa B)^{c}\cap E_{\xi })^c}\leq r^{2})
\leq \mathbb{P}_y (\tau_{(F^c)} \leq r^2) \leq 
2\mathrm{%
cap}(F, \Omega)\int_{1}^{r^{2}}\sup_{\omega \in \Omega\backslash
F}p(s, \omega ,y)ds,
\]%
where we chose $\kappa>2$ and 
\begin{small}
\begin{eqnarray*}
F=\left\{ \omega \in E_{\xi } ~:~ \frac{3\kappa r}{4} \leq |\omega | \leq \frac{5\kappa r}{4}
\right\}, ~ \Omega =\left\{ \omega \in E_{\xi } ~:~ \frac{\kappa r}{2} < |\omega | < \frac{%
3\kappa r }{2} \right\}.
\end{eqnarray*}
\end{small}
Here we remark that $(\partial \kappa B) \cap E_{\xi } \subset F  \subset \Omega$ and 
$ y \not \in \Omega $ because of $\kappa >2$. 

By the estimate in Theorem \ref{off-diagonal}, 
\begin{small}
\begin{eqnarray*}
&&\!\!\!\!\!\!\!\!\!\!\!\!\! \sup_{\omega \in \Omega \backslash F} p(s, \omega  ,y)\asymp 
 C\Big[ \frac{\delta
_{\xi j}}{V_{\xi }(\sqrt{s})} \frac{h_{j}(|y|)}{h_{j}(|y|)+h_{j}(\sqrt{s})} 
+p(s,o,o)\frac{r^{2}}{V_{\xi }(r)h_{\xi }(r)}\mathbb{P}_{y}(\tau _{K}\leq s) \\
&&\!\!\!\!\!\!\!\!\!\!\!\!\!  +\!\! \int_{1}^{s} \!\!\! p(u,o,o)du \Big( \frac{1}{V_{\xi}(\sqrt{s})h_{\xi}(\sqrt{s})}%
\mathbb{P}_{y}(\tau _{K}\leq s)
+\frac{r^{2}}{V_{\xi}(r)h_{\xi}(r)}\partial _{s}%
\mathbb{P}_{y}(\tau _{K}\leq s)\Big) \Big]e^{-b(\kappa r)^{2}/s}.
\end{eqnarray*}%
\end{small}
Then we obtain 
\begin{small}
\begin{eqnarray*}
\int_{1}^{r^{2}} \sup_{\omega \in \Omega \backslash F} p(s, \omega ,y)ds &&\\
&&\!\!\!\!\!\!\!\!\!\!\!\!\!\!\!\!\!\!\!\!\!\!\!\!\!\!\!\!\!\!\!\!\!\!\! \!\!   \leq 
C \Big[\delta _{\xi j}\frac{r^{2}}{V_{\xi }(r)}\frac{h_{j}(|y|)}{h_{j}(r)}%
+\int_{1}^{r^{2}}p(s,o,o)ds\frac{r^{2}}{V_{\xi }(r)h_{\xi }(r)}\mathbb{P}_{y}(\tau
_{K}\leq r^{2})\Big]e^{-b\kappa ^{2}}.
\end{eqnarray*}
\end{small}
Since 
\begin{small}
\[
\mathrm{cap}(F, \Omega)=
\mathrm{cap}\Big( \overline{B_{\xi } \Big( o_{\xi }, \frac{\kappa r}{2} \Big) }, 
B_{\xi } \Big( o_{\xi }, \frac{3\kappa r}{4} \Big) \Big)
+
\mathrm{cap}\Big( \overline{B_{\xi } \Big( o_{\xi }, \frac{5\kappa r}{4} \Big) }, 
B_{\xi } \Big( o_{\xi }, \frac{3\kappa r}{2} \Big) \Big),
\]
\end{small}
where $B_{\xi } (x,r)$ is a geodesic open ball in $M_{\xi }$. 
By using the general estimate for any $0<r<R$ 
\[
\mathrm{cap}(\overline{B(o,r)}, B(o,R)) \leq \frac{2}{ \int_r^R \frac{sds}{\mu(B(o,s))}}
\]

(see \cite[Theorem 7.1]{G 1999}, \cite[(4.5)]{G-SC Dirichlet} for the detail), 
we obtain
\[
\mathrm{cap}(F, \Omega) \leq C\frac{V_{\xi }(r)}{r^{2}}.
\]%
By Theorem 4.6 in \cite{G-SC hitting}, 
\[
\mathbb{P}_{y}(\tau _{K}\leq r^{2}) \leq \frac{C}{h_{j}(r)}\int_{|y|}^{r}%
\frac{sds}{V_{j}(s)}.
\]%
Combining these estimate and Lemma \ref{technical}, we conclude 
(\ref{step2}).


\textbf{Step3: Estimate of }$\sup_{0<s\leq r^2, z \in( \partial \kappa
B)\cap E_{\xi }}p(s, z ,x)$

We will prove that for any $x \in A_i(r)$
\begin{equation}
\sup_{\stackrel{0<s\leq r^2}{ z \in (\partial \kappa B)\cap E_{\xi }}} \!\!p(s,z ,x)\leq \frac{C%
}{\kappa^{2\alpha}}\Big[\delta _{i\xi }\frac{1}{V_{i}(r)}+\min_{\eta} h_{\eta}  (r) \frac{r^{2}}{%
V_{i}(r)h_{i}(r)}\frac{1}{V_{\xi }(r)h_{\xi }(r)}\Big].  \label{step3}
\end{equation}

By the estimates in Theorem \ref{off-diagonal}, 
for all $0<s\leq r^2$ and $z  \in (\partial \kappa B)\cap E_{\xi }$
\begin{small}
\begin{eqnarray*}
&&p(s,x, z )\asymp  Ce^{-b(\kappa r)^{2}/s}\Big[\delta _{i\xi }\frac{1}{V_{i}(\sqrt{s})}+p(s,o,o)\frac{%
r^{2}}{V_{i}(r)h_{i}(r)}\frac{r^{2}}{V_{\xi }(r)h_{\xi }(r)} \\
&& +\int_{1}^{s}p(u,o,o)du\left( \frac{1}{V_{i}(\sqrt{s})h_{i}(\sqrt{s})}%
\frac{r^{2}}{V_{\xi }(r)h_{\xi }(r)}+\frac{r^{2}}{V_{i}(r)h_{i}(r)}\frac{1}{V_{\xi }(%
\sqrt{s})h_{\xi }(\sqrt{s})}\right) \Big].
\end{eqnarray*}%
\end{small}
Since $V_i$ and $\widetilde{V}_i$ are doubling, 
there exist $C, \alpha>0$ such that 
\begin{eqnarray*}
\frac{1}{V_i(\sqrt{s})} &\leq &\frac{C}{V_i(r)} \left( \frac{r^2}{s}
\right)^{\alpha}, \\
p(s, o,o) &\simeq &p(r^2, o ,o) \frac{ \frac{\min_i h_i^2 (\sqrt{s})}{\min_i 
\widetilde{V}_i (\sqrt{s})}} { \frac{\min_i h_i^2 (r)}{\min_i \widetilde{V}%
_i (r)}} \\
&\leq &p(r^2, o, o) \frac{ \min_i \widetilde{V}_i(r)}{\min_i 
\widetilde{V}_i(\sqrt{s})} \leq Cp(r^2, o, o) \left( \frac{r^2}{s}
\right)^{\alpha} , \\
\frac{1}{V_i(\sqrt{s})h_i(\sqrt{s})}& = & \frac{h_i(\sqrt{s})}{\widetilde{V}%
_i(\sqrt{s})} \leq \frac{C}{V_i(r) h_i(r)} \left( \frac{r^2}{s}
\right)^{\alpha}.
\end{eqnarray*}
Because the inequality
\[
t^{\alpha} e^{-t} \leq \alpha^{\alpha} e^{-\alpha}
\]
is always true for all $t>0$, we have
\[
\left( \frac{r^2}{s} \right)^{\alpha} e^{-b\kappa^2 \frac{r^2}{s}} \leq 
\frac{1}{b^{\alpha} \kappa^{2\alpha}} \alpha^{\alpha} e^{-\alpha}.
\]
Then we obtain 
\begin{small}
\begin{eqnarray*}
\sup_{\stackrel{0<s\leq r^2}{ z \in (\partial \kappa B) \cap E_{\xi }}} \!\! p(s, z ,x) \!\! &\leq&  \!\!
\frac{C}{\kappa^{2\alpha}}\Big[\delta _{i\xi }\frac{1}{V_{i}(r)}+p(r^2,o,o)\frac{r^{2}}{%
V_{i}(r)h_{i}(r)}\frac{r^{2}}{V_{\xi }(r)h_{\xi }(r)} \\
&&\!\!\!\!\!\!\!\!\!\!\!\!\! \!\!\!\!\!\!\!\!\!\!\!\!\!\!\!\!\!\!\!\!\!\!\!\!\!\!\!\!\!\! +\int_{1}^{r^2}p(u,o,o)du\left( \frac{1}{V_{i}(r)h_{i}(r)}\frac{r^{2}}{%
V_{\xi }(r)h_{\xi }(r)}+\frac{r^{2}}{V_{i}(r)h_{i}(r)}\frac{1}{V_{\xi } (r)h_{\xi }(r)}%
\right) \Big].
\end{eqnarray*}%
\end{small}
Using (\ref{trivial}) and Lemma \ref{technical} again, we conclude (\ref{step3}).


\textbf{Step4: Proof of (\ref{estimate Dirichlet})}

Combining (\ref{step2}) and (\ref{step3}), we obtain 
\begin{small}
\begin{eqnarray*}
\sum_{\xi =1}^{k}&&\!\!\!\!\!\!\!\!\!\!\!\!\!\!\! \mathbb{P}_{y}(\tau _{(\kappa B)^{c}\cap E_{\xi }}\leq
r^{2}) \sup_{\stackrel{0<s\leq r^{2}}{z \in (\partial \kappa B) \cap E_{\xi }}}p(s, z ,x) \\
&\leq &\frac{Ce^{-b\kappa^2}}{\kappa^{2\alpha}} \sum_{\xi =1}^{k}\left( \delta _{\xi j}\frac{h_{j}(|y|)}{%
h_{j}(r)}+\frac{\min_{\eta} h_{\eta} (r)}{h_{\xi }(r)}\frac{1}{h_{j}(r)}\int_{|y|}^{r}%
\frac{sds}{V_{j}(s)}\right) \\
&&\quad\qquad\qquad\times \left( \delta _{i\xi }\frac{1}{V_{i}(r)}+\min_{\eta} 
h_{\eta} (r) \frac{r^{2}}{V_{i}(r)h_{i}(r)}\frac{1}{V_{\xi }(r)h_{\xi }(r)}\right) \\
&=&  \frac{Ce^{-b\kappa^2}}{\kappa^{2\alpha}}\frac{1}{V_{i}(r)h_{j}(r)}\sum_{\xi =1}^{k}\left( \delta
_{\xi j}h_{j}(|y|)+\frac{\min_{\eta} h_{\eta} (r) }{h_{\xi }(r)}\int_{|y|}^{r}\frac{sds}{%
V_{j}(s)}\right) \\
&&\qquad\qquad\qquad\qquad \qquad \times \left( \delta _{i\xi }+\frac{\min_{\eta} h_{\eta} (r) }{h_{i}(r)}\frac{%
r^{2}}{V_{\xi }(r)h_{\xi }(r)}\right) \\
&=&  \frac{Ce^{-b\kappa^2}}{\kappa^{2\alpha}}\frac{1}{V_{i}(r)h_{j}(r)}\Big[\delta _{ij}h_{j}(|y|) \\
&&\quad+%
\frac{\min_{\eta} h_{\eta} (r) }{h_{i}(r)}\Big( \int_{|y|}^{r}\frac{sds}{V_{j}(s)}%
\Big( 1+\frac{r^{2}\min_{\eta} h_{\eta} (r) }{\min_{\eta}\widetilde{V}_{\eta}(r)}\Big) +%
\frac{r^{2}}{V_{j}(r)h_{j}(r)}h_{j}(|y|)\Big) \Big].
\end{eqnarray*}
\end{small}

By using Lemma \ref{technical} and a trivial lower bound $V_{\eta} (r)h_{\eta} (r)
\gtrsim r^2$, we obtain 
\[
r^2 \min_{\eta} h_{\eta} (r) \lesssim \min_{\eta} \widetilde{V}_{\eta}(r)
\]
which implies that 
\[
1+\frac{r^{2}\min_{\eta} h_{\eta} (r) }{\min_{\eta}\widetilde{V}_{\eta}(r)} \simeq 1.
\]
Then we have 
\begin{small}
\begin{eqnarray}
 \sum_{\xi =1}^{k}\mathbb{P}_{y}(\tau _{(\kappa B)^{c}\cap E_{\xi }}\leq
r^{2})&&\!\!\!\!\!\!\!\!\!\!\!\!\! \sup_{\stackrel{0<s\leq r^{2}}{z \in (\partial \kappa B)\cap E_{\xi }}}p(s, z ,x) 
\nonumber \\
&& \!\!\!\!\!\!\!\!\!\!\!\!\!\!\!\!\!\!\!\!\!\!\!\!\!\!\!\!\!\!\!\!\!\!\!\!\!\!\!\!\!\!\!\!\!\!\!\!\!\!\!\!\!\!\!\!\!\!\!\!\!\!\!\!\!
\!\!\!\!\!\!
\leq \frac{Ce^{-b\kappa^2}}{\kappa^{2\alpha}} \frac{1}{V_{i}(r)h_{j}(r)}\Big[\delta
_{ij}h_{j}(|y|) 
+\frac{\min_{\eta} h_{\eta} (r)}{h_{i}(r)}\Big( \int_{|y|}^{r}\frac{sds%
}{V_{j}(s)}+\frac{r^{2}}{V_{j}(r)h_{j}(r)}h_{j}(|y|)\Big) \Big], \nonumber \\
\label{step 4}
\end{eqnarray}%
\end{small}
which is same as  (\ref{step1-2}) up to $Ce^{-b\kappa^2}/\kappa^{2\alpha}$.
Choosing $\kappa >2$ large enough and substituting (\ref{step1-2}) and (\ref{step 4}) into 
(\ref{Dirichlet decomp}), we conclude (\ref{estimate Dirichlet}).


\textbf{Step5: Proof of (\ref{lower Dirichlet})}

If $i=j$, then we obtain by the
estimate in (\ref{estimate Dirichlet}) and the fact that  $h_j(|y|)+\int_{|y|}^r \frac{sds}{V_j(s)} =h_j(r)$
\begin{equation}
p_{\kappa B}^{D}(r^{2},x,y)\geq \frac{c}{V_i(r) h_i(r)} \frac{\min_{\eta} h_{\eta} (r)}{h_i(r)} h_i(r)
=c\frac{\min_{\eta} h_{\eta} (r)}{V_i(r)h_i(r)}.\label{parab 1}
\end{equation}
If $i\neq j$, then the function 
\[
F(t)=\int_t^{r} \frac{sds}{V_j(s)} +\frac{Cr^2}{V_j(r)h_j(r)} h_j(t)
\]
is monotone decreasing for some constant $C>0$ and for all $r \gg 1$, $t \simeq r$ 
gives the lower bound of $F$, namely we obtain
\[
F(r) \geq c \frac{r^2}{V_j(r)}.
\]
Then we obtain by the estimate in (\ref{estimate Dirichlet})
\begin{eqnarray}
p_{\kappa B}^{D}(r^{2},x,y) &\geq& \frac{c}{V_i(r)h_j(r)} \frac{\min_{\eta} h_{\eta} (r)}{h_i(r)} \frac{r^2}{V_j(r)}
\nonumber\\
 & = &c\frac{ \min_{\eta} h_{\eta} (r)}{ V_i(r)h_i(r)} \frac{r^2}{V_j(r) h_j(r)}.
\label{parab 2}
\end{eqnarray}

%
Because $V_j(r) h_j(r) \geq c r^2$ is always true, combining estimates in (\ref{parab 1}) and (\ref{parab 2}), we conclude the
estimate in (\ref{lower Dirichlet}).
\end{proof}


\subsection{Upper bound on Poincar\'e constant}

\label{SecUBPoincare}We prove here the upper bounds of the Poincar\'{e}
constant: (\ref{PC upper non-parab}) of Theorem \ref{main non-parab} and (%
\ref{PC upper parab}) of Theorem \ref{main parab} with some large $\kappa \ge 1$. 
The following lemma inspired by Kusuoka-Stroock \cite[Theorem 5.10]{KS} (see also \cite[5.5.1]{SC LNS}) 
plays a key role to obtain an upper bound of the Poincar\'e constant from a lower bound 
of the Dirichlet heat kernel.
\begin{lemma}\label{Dirichlet Poincare}
Let $(M,\mu)$ be a weighted manifold. 
Let $U \subset U^\prime \subset M$ be precompact connected open sets. 
Then for any $t>0$ and any open set $A\subset U^\prime$, 
\[
\Lambda(U, U^\prime) \leq \frac{2t}{\mu(A) \inf_{x \in A, y \in U} 
p_{U^\prime}^D(t, x, y)}.
\]
\end{lemma}
\begin{proof} Let $\{P_{t}^{N, U^\prime} \}$ 
be the Neumann-heat semigroup on $U^\prime $ and 
$p_{U^\prime }^{N}(t,x,y)$ its kernel function (Neumann heat kernel). 
By a well-known fact for Dirichlet forms, for any $f\in
C^{1}(\overline{U^\prime} )$, and for any $t>0$, 
\[
\int_{U^\prime }|\nabla f|^{2}d\mu =\mathcal{E}(f,f)_{U^\prime }\geq \left( 
\frac{f-P_{2t}^{N, U^\prime }f}{2t},f\right) _{L^{2}(U^\prime )}.
\]%
Then we obtain for any open set $A \subset U^\prime$
\begin{eqnarray}
2t\int_{U^\prime }|\nabla f|^{2}d\mu & \geq &
\left( f-P_{2t}^{N, U^\prime }f,f\right) _{L^{2}(U^\prime )}  \nonumber\\
& =& \int_{U^\prime }P_{t}^{N, U^\prime }\left[ f-P_{t}^{N, U^\prime }f(x)%
\right] ^{2}(x)d\mu (x) \nonumber\\
& \geq &\int_{A}P_{t}^{N, U^\prime }\left[ f-P_{t}^{N, U^\prime }f(x)\right] ^{2}(x)d\mu (x) 
\nonumber\\
& \geq &\mu (A)\inf_{x\in A}P_{t}^{N, U^\prime }\left[
f-P_{t}^{N, U^\prime }f(x)\right] ^{2}(x). 
\label{KS1}
\end{eqnarray}%
Because the Neumann heat kernel is always larger than the Dirichlet heat
kernel, we obtain for all $x\in A$ 
\begin{small}
\begin{eqnarray}
P_{t}^{N, U^\prime }\left[ f-P_{t}^{N, U^\prime }f(x)\right] ^{2}(x)& 
=&  \int_{U^\prime }p_{U^\prime }^{N}(t,x,y)\left[ f(y)-P_{t}^{N, U^\prime }f(x)\right] ^{2}d\mu (y)  \nonumber \\
& \geq &  \int_{U }p_{U^\prime }^{D}(t,x,y)\left[ f(y)-P_{t}^{N, U^\prime }f(x)\right] ^{2}d\mu (y)  \nonumber \\
& \geq &\inf_{y\in U}p_{U^\prime }^{D}(t,x,y)\int_{U}|f(y)-P_{t}^{N, U^\prime }f(x)|^{2}d\mu (y) 
\nonumber \\
& \geq& \inf_{y\in U}p_{U^\prime }^{D}(t,x,y)\int_{U}|f(y)-f_{U}|^{2}d\mu (y).  \label{integrand}
\end{eqnarray}%
\end{small}
Here the final line follows from the fact that $f_{U}$ minimizes the
function $\xi \mapsto \int_{U}|f-\xi |^{2}d\mu $. 
Combining the estimates in (\ref{KS1}) and (\ref{integrand}), 
we complete the proof.
\end{proof}

\begin{proof}[Proof of Theorems \ref{main non-parab} and \ref{main parab}]
Now we start to prove the upper bound of the Poincar\'e constant in the 
main theorems. For the lower bound, see section \ref{SecLBPoincare}.

Recall that $M$ is the manifold with nice ends $M_1, \ldots , M_k$,
where each parabolic end satisfies (RCA).
If there is at least one parabolic end, we assume also that $M$ admits (COE) 
defined in Definition \ref{DefCOE}. 
By the results in Theorems \ref{non-parabolic} and \ref{main heat},
 $M$ satisfies the estimate in (\ref{min min}) which allows us to use 
Lemma \ref{Dirichlet}. 
Let $\kappa>1  $ be the constant from (\ref{estimate Dirichlet}%
). Apply Lemma \ref{Dirichlet Poincare} to the case of $U=B=B(o,r)$, 
$U^\prime=\kappa B=B(o, \kappa r)$, $t=r^2$ and $A=A_i(r)$ given in (\ref{annulus}).
Then we obtain
\[
\Lambda(B, \kappa B) 
\leq \frac{2r^2}{\mu(A_i(r)) \inf_{x \in A_i(r), y \in B} 
p_{\kappa B}^D(r^2, x, y)}.
\]
Since $\mu(A_i(r)) \simeq V_i(r)$ by 
(VD) on $M_i$, substituting the estimates in (\ref{lower Dirichlet})
to the above, we obtain for all $i=1,\ldots, k$
\begin{equation}
\Lambda(B, \kappa B)\lesssim \frac{h_i(r) }{
 \min_{\eta} h_{\eta} (r)}\max_{j \neq i} \left\{ V_j(r) h_j(r) \right\}.
\label{PC general}
\end{equation}

If all ends are non-parabolic, then $h_i(r)\simeq 1$ for all $i=1,  \ldots ,k$.
By taking $i=m=m(r)$, the index of the largest end at $r$, we obtain
\[
\Lambda(B, \kappa B) \lesssim \max_{j \neq m} V_j(r) \simeq
V_n(r).
\]
Hence, the estimate in (\ref{PC upper non-parab}) holds for some $\kappa \ge
1$. 

Next, let us assume that $M=\#_{i\in I}M_{i}$ is a manifold with (COE).
Assume first that $I_{super}\neq \emptyset$. Then the index of the 
largest end $m$ is in $I_{super}$ and $\min_{\eta} h_{\eta} (r) \simeq 1$. Taking $i=m$,
we obtain
\[
\Lambda (B, \kappa B) \lesssim \max_{ j \neq m} \left\{ V_j(r)h_j(r) \right\}.
\]
By the second condition in (\ref{(ii)-1}) and the definition of (COE), we
conclude 
\[
 \max_{j\neq m} \left\{ V_{j}(r)h_{j}(r) \right\} \simeq
V_{n}(r)h_{n}(r),
\]
which gives the estimate in (\ref{PC upper parab}) for some $\kappa \geq 1$.

If $I_{super}=\emptyset$ and $I_{middle} \neq \emptyset$ 
then, by Lemma \ref{Lem2} there exists 
a dominating volume function $V_{l}$ satisfying (\ref{dominate}) with $l \in I_{middle}$ and
 $\min_i h_i (r) \simeq h_{l} (r)$.
Estimate (\ref{PC general}) with $i=l$ implies that
\[
\Lambda (B, \kappa B) \lesssim \max_{j \neq l} \left\{ V_j(r) h_j(r) \right\}
\simeq V_n(r)h_n(r),
\]
which gives the estimate in (\ref{PC upper parab}) for some $\kappa \geq 1$.

Finally, assume that $I_{super}=I_{middle}=\emptyset$, that is, all ends are subcritical. 
By the definition of subcriticality given in Definition \ref{DefCOE} (b), 
\[
\min_{\eta} h_{\eta} (r) \lesssim \min_{\eta} \frac{r^2}{V_{\eta} (r)} =\frac{r^2}{V_m(r)}.
\]
Taking $i=m$ in the estimate in (\ref{PC general}),  we obtain
\[
\Lambda(B, \kappa B )\lesssim \max_{j \neq m} \left\{ V_j (r) h_j(r) \right\} 
\simeq  r^2, 
\]
which completes the proof of the estimate in(\ref{PC upper parab}) with $%
\kappa \geq 1$. 

By using  an additional argument in Section \ref{section Whitney}, 
we will show  that we can reduce $\kappa >1$ to $\kappa=1$.
\end{proof}

\subsection{Lower bound on the Poincar\'e constant}

\label{SecLBPoincare}
In this section we prove the matching lower bound for the Poincar\'{e} constant
under a different hypothesis.

\begin{theorem}
\label{Pmain(c)} Let $M=M_1\# \cdots \# M_k$ be a manifolds with ends such
that each end $M_i$ satisfies {\rm{(VD)}} and 
\begin{equation}
CV_{i}\left( r\right) \geq rV_{i}^{\prime }\left( r\right) \ \ \mbox{for all 
} r\ge r_0.  \label{VrV'}
\end{equation}%
Then for all large $r\gg 1$ 
\begin{equation}
\Lambda (B(o,r)) \geq c \max_{i\neq m}\left\{ V_{i}(r)h_{i}(r) \right\},
\label{general lower}
\end{equation}
where $m=m(r)$ is the index of the largest end at scale $r$, namely for all $%
i \in I$ 
\begin{equation}
V_m(r) \ge V_i(r).  \label{V max}
\end{equation}
\end{theorem}

Let $n=n(r)$ be the index of the second largest end. Then the estimate in (%
\ref{general lower}) implies that 
\begin{equation}
\Lambda(B(o,r)) \geq c V_n(r)h_n(r).  \label{lower second}
\end{equation}
In view of the upper estimate of the Poincar\'e constant in Theorems \ref%
{main non-parab} and \ref{main parab}, the lower bound in (\ref{lower second}%
) is optimal if either all ends are non-parabolic, or $M$ is a manifold with
{\rm{(COE)}}.

\begin{proof}
Set $B=B(o,r)$. First we show that for any $\varepsilon \in (0,1)$, 
\begin{equation}
\varepsilon \sup_{\stackrel{ f\in C^{1}(B)}{\mu (B\cap \{f=0\})\geq
\varepsilon \mu (B) }}\frac{\int_{B}|f|^{2}d\mu }{\int_{B}|\nabla f|^{2}d\mu 
}\leq \Lambda (B,B).  \label{zero}
\end{equation}%
Indeed, for any $f\in C^{1}(B)$ with $\mu (B\cap \{f=0\})\geq \varepsilon
\mu (B)$, 
\begin{eqnarray*}
\int_{B}|f-f_{B}|^{2}d\mu& =& \int_{B}|f|^{2}d\mu -\mu (B)f_{B}^{2} \\
&=& \int_{B}|f|^{2}d\mu -\frac{1}{\mu (B)}\left( \int_{B\cap \{f\neq
0\}}fd\mu \right) ^{2}.
\end{eqnarray*}%
Cauchy-Schwarz inequality implies that 
\begin{eqnarray*}
\int_{B}|f-f_{B}|^{2}d\mu &\geq & \int_{B}|f|^{2}d\mu -\frac{\mu (B\cap
\{f\neq 0\})}{\mu (B)}\int_{B}|f|^{2}d\mu \\
&=& \frac{\mu (B\cap \{f=0\})}{\mu (B)}\int_{B}|f|^{2}d\mu \geq \varepsilon
\int_{B}|f|^{2}d\mu .
\end{eqnarray*}%
By the expression of the Poincar\'{e} constant (\ref{Lam sup}), we obtain
the estimate in (\ref{zero}).

Now we prove the estimate in (\ref{general lower}) by choosing a test
function. For $i \neq m$, take a $C^1$-function $f_i$ defined by 
\[
f_i(x)= \left\{ 
\begin{array}{ll}
h_i(|x|) & x \in B\cap E_i , \\ 
0 & x \in B\cap E_j, ~ j\neq i.%
\end{array}
\right.
\]
We note that the assumption (\ref{V max}) implies that
\[
\mu(B\cap E_m) \geq \frac{1}{2}\mu(B \cap E_j)
\]
for large $r \gg 1$ and all $j \in\{ 1 ,\ldots , k \}$. Because $f_i=0 $ in $B\cap E_m$, we
observe that 
\begin{eqnarray*}
\mu(B\cap \{f_i= 0 \}) &\geq & \mu(B \cap E_m) \\
&= & \frac{1}{2k}\left( 2k\mu(B\cap E_m) +\mu(K) \right) -\frac{1}{2k} \mu(K)
\\
&\ge & \frac{1}{2k}\left( \mu(B) - \mu(K) \right).
\end{eqnarray*}
Then for large $r\gg 1$ so that $\mu(B) \ge 2\mu(K)$, we obtain 
\[
\mu(B\cap \{ f=0 \}) \geq \frac{1}{4k} \mu(B).
\]
Hence, (\ref{zero}) implies that 
\[
\frac{1}{4k} \frac{ \int_{B}|f_i|^2 d\mu}{\int_{ B} |\nabla f_i|^2 d\mu}
\leq \Lambda(B, B).
\]
By using the assumptions (VD) and (\ref{VrV'}), we obtain for large $r\gg 1$ 
\begin{eqnarray*}
\int_B |f_i |^2 d\mu& \geq & \int_{(B\backslash \frac{1}{2}B)\cap E_i}
|f_i|^2 d\mu \\
&\ge &\left( \mu(B\cap E_i) - \mu( \frac{1}{2}B\cap E_i) \right)
\inf_{(B\backslash \frac{1}{2}B) \cap E_i } |f_i|^2 \\
&\geq &cV_i(r) h_i^2(r), \\
\int_{ B} |\nabla f_i |^2 d\mu &=& \int_{B(o, r_0)} |\nabla f_i |^2 d\mu +
\int_{r_0}^{ r} \left| \partial_s h_i (s) \right|^2 V^\prime (s) ds \\
&=& C +\int_{r_0}^{ r} \frac{s}{V_i(s)} \frac{sV^\prime (s)}{V(s)} ds
\leq C+Ch_i(r) \\
&\leq & C^\prime h_i( r).
\end{eqnarray*}
Then we obtain for any $i\neq m$, 
\[
cV_i(r)h_i(r) \leq \Lambda (B)
\]
which concludes (\ref{general lower}).
\end{proof}

\section{Appendix: Spectral gap on central balls}
\setcounter{equation}{0}
\label{section Whitney}
Recall that $M$ is a connected sum of manifolds $M_1, \ldots , M_k$ with a central
part $K$. 
The purpose of this section is to obtain a general upper
bound of the Poincar\'e constant $\Lambda(B(o,r))=\Lambda(B(o, r), B(o,r))$ at the central reference 
point $o \in K$ for any large $r>0$ by using a collection of $\Lambda(B(x, s), B(x,
\kappa s))$, where $B(x, s) \subset B(o, r)$ and $\kappa \geq 1$. This is an 
important procedure to obtain a lower bound of the spectral gap $\lambda (B(o,r))$ of $-\Delta$ on $B(o,r)$ defined in (\ref{ev}).
%
First of all, we only assume the volume doubling condition
(VD) on each $M_i$. 

The main tool to obtain a desired bound is a \textit{Whitney covering} $\mathcal{W}$ of 
$B(o,r)$ which is a collection of balls defined as the following.

\begin{definition}
For any $0<\eta < 1$, a collection of balls $\mathcal{W}=\mathcal{W}(\eta)$ in $B(o,r)$
is called a Whitney covering of $B(o,r)$ with parameter $\eta$ if
\begin{enumerate}
\item[(W1)] All $ F\in \mathcal{W}$ are disjoint,
\item[(W2)] $\cup_{F \in \mathcal{W}} 3F =B(o,r)$,
\item[(W3)] $  r(F)=\eta d(F, B(o,r)^c)$. Here $r(F)$ is the radius of $F$ and $B(o,r)^c$ is 
the complement of $B(o,r)$.
\item[(W4)] For any $\alpha\geq 1$, there exists $N=N(\eta, \alpha)$ 
independent of $r$ such that for any $x \in B(o,r)$
\[
               \# \{ F \in \mathcal{W} ~:~ x \in \alpha F \} \leq N.
\]
\end{enumerate}
\end{definition}
It is a well-known fact that there exists such a covering. 
For $F \in \mathcal{W}$, we denote by $\gamma_F$ a distance-minimizing curve 
joining the center of $F$ and $o$. Let $\mathcal{F}(F)=
(F_0, F_1, \ldots , F_{l(F)})$ be a string of $F$, that is, a sequence of balls in $\mathcal{W}$ 
so that $3F_j \cap 3F_{j+1} \neq \emptyset$ ($j=0, 1, \ldots , l(F)-1$), 
$o \in 3F_0$ and $F_{l(F)}=F$. It is known that there exists such a string 
$\mathcal{F}(F)$ for any $F \in \mathcal{W}$. See \cite[Section 5.3.3]{SC LNS} for the detail.
\begin{proposition}[c.f. \mbox{\cite[Section 5.3.3]{SC LNS}} ]
Let $\mathcal{W}$ be a Whitney covering of $B(o,r)$ with parameter $0<\eta< 1/4$. 
Then $\mathcal{W}$ satisfies the following. 
\begin{enumerate}
\item[{\rm{(PW1)}}] 
(  \cite[Lemma 5.3.6]{SC LNS}) 
For any $F, F^\prime \in \mathcal{W}$ so that $3F^\prime \cap \gamma_{F} \neq \emptyset$, 
\[
r(F^\prime) \geq \frac{1}{4\eta +1} r(F).
\]

\item[{\rm{(PW2)}}] 
(\cite[Lemma 5.3.7]{SC LNS}) 
For $F \in \mathcal{W}$, let $\mathcal{F}(F)=
(F_0, F_1, \ldots , F_{l(F)})$ be a string of $F$. Then for any $j=0, 1, \ldots ,l(F)-1$, 
\begin{eqnarray*}
(\eta^{-1}-4)r(F_j) &\leq & (4+\eta^{-1})r(F_{j+1}) \\
(\eta^{-1}-4)r(F_{j+1}) &\leq & (4+\eta^{-1}) r(F_j),
\end{eqnarray*}
and 
\begin{equation}
3F_{j+1} \subset \left( 6\frac{\eta^{-1}+4}{\eta^{-1}-4} +3 \right) F_j.
\label{Fj}
\end{equation}
\end{enumerate}
\end{proposition}
We note that $6\frac{\eta^{-1}+4}{\eta^{-1}-4} +3 <12 $ if $\eta<1/20$. 
Thus, we always assume that $\eta <1/20$ in the sequel.

\begin{lemma}\label{lemB0}
For any $\kappa \geq 1$ and any $0<\eta<1/20$, let $B_0=B(o, 12\kappa \eta r)$. 
Then for  any $r$ large enough so that 
$\mathrm{diam}K \leq \eta r$ and for any $F \in  \mathcal{W}$ so that 
\begin{equation}
F \cap B_0 =\emptyset,
\label{B0}
\end{equation}
we have 
\[
12\kappa F \cap K=\emptyset.
\]
\end{lemma}
\begin{proof}
Since $r(F) \leq \eta r$ by (W3), the condition (\ref{B0}) implies that 
\begin{eqnarray*}
d(o, o(F)) &\geq &r(F) +r(B_0)=r(F)+12 \kappa  \eta r\\
&=& r(F) +(12\kappa -1)\eta r + \eta r \\
&\geq & 12\kappa  r(F) + \mathrm{diam} K,
\end{eqnarray*}
which concludes the lemma.
\end{proof}

In the sequel, we always assume that $r$ is large enough so that 
\begin{equation}
\mathrm{diam}K \leq \eta r.
\label{eta r}
\end{equation}

\begin{lemma}\label{lem3B0}
For $F\in \mathcal{W}$, if $3F \not \subset 3B_0$, then $F \cap B_0=\emptyset$.
\end{lemma}
\begin{proof}Suppose that $F\cap B_0 \neq \emptyset$. Then 
for any $x \in 3F$ and $z \in F \cap B_0$, 
\begin{eqnarray*}
d(o, x) &\leq &d(o, z)+d(z, x) \\
& \leq &r(B_0) +4r(F) \\
&\leq &12\kappa \eta r +4 \eta r \leq 3r(B_0).
\end{eqnarray*}
This implies that $3F \subset 3B_0$. 
By contraposition, we conclude the lemma.
\end{proof}

\begin{lemma}\label{lemF1}
For $F \in \mathcal{W}$, if $3F \cap 3B_0 \neq \emptyset$, then 
\[
3F \subset B_1=B(o,(36\kappa  +6) \eta r).
\]
\end{lemma}
\begin{proof}
For $x \in 3F$ and $z \in 3F \cap 3B_0$, we obtain
\begin{eqnarray*}
d(o, x) &\leq & d(o, z) + d(z , x)\\
&\leq & 36\kappa \eta r + 6r(F) \\
&\leq & (36 \kappa  + 6) \eta r,
\end{eqnarray*}
which concludes the lemma.
\end{proof}

Now we modify the Whitney covering $\mathcal{W}=\mathcal{W}(\eta)$
 to fit the manifolds with ends in 
concern. 
Set 
\[
\mathcal{W}^\prime =\{ F \in \mathcal{W} ~:~ 3F \not \subset 3B_0 \}.
\]
Here we note that 
\begin{equation}
B(o,r)=\cup_{F \in \mathcal{W}^\prime} 3F \cup 3B_0.
\label{W'}
\end{equation}
Moreover, for $i=1,\ldots, k$, set 
$\mathcal{W}_i^\prime=\{ F \in \mathcal{W}^\prime ~:~ F \subset E_i \}$. 
We can easily see that they are disjoint each other by the definition, and,  Lemmas 
\ref{lemB0} and \ref{lem3B0} imply that
\[
\mathcal{W}^\prime =\cup_{i=1}^k \mathcal{W}_i^\prime.
\]
For $F\in \mathcal{W}_i^\prime$, we retake a string $\mathcal{F}(F)
=(F_0, F_1, \ldots, F_{l^\prime (F)})$ given in (PW2) so that $F_j \subset E_i$ 
($j\neq 0$) and $F_0=B_0$. 

The following  inclusion conditions play a key role in the proof of 
Theorem \ref{main Whitney}.
\begin{lemma}For a constant $\kappa \geq 1$, let $\eta>0$ be a parameter satisfying
\begin{equation}
\eta<
\frac{1}{\kappa(36\kappa +6)} 
.
\label{eta}
\end{equation}
For such $\eta>0$, let $r>0$ be a number 
so that (\ref{eta r}) holds. Then we have
\begin{eqnarray}
3\kappa B_0 &\subset& \kappa B_1\subset B(o,r),  \label{inc1} \\
12\kappa F & \subset & B(o, r) \cap E_i \quad (F\in \mathcal{W}_i^\prime) 
\label{inc2}.
\end{eqnarray}
\end{lemma}
The main result of this section is the following.
\begin{theorem}\label{main Whitney}
Let $M$ be a connected sum of manifolds $M_1, \ldots , M_k$ with a compact central part $K$, where each $M_i$ admits {\rm{(VD)}}. 
Let $o \in K$ be a central reference point. 
Fix $\kappa \geq 1$. Fix $\eta>0$ satisfying (\ref{eta}).
For any $r>0$ satisfying $\mathrm{diam}K \leq \eta r$, 
there exists a constant $C>0$ 
such that
\begin{eqnarray*}
&&\!\!\!\!\!\!\!\!\!\!\! \Lambda (B(o,r))  \\
&&\!\!\!\!\!\!\!\!\!\!\! \leq C\left( 
\Lambda(3B_0, 3\kappa B_0)+ \Lambda(B_1, \kappa B_1)+\max_{F \in \mathcal{W}^\prime(\eta)} 
\{ \Lambda( 3F, 3\kappa F), \Lambda(12F, 12\kappa  F) \}
\right),
\end{eqnarray*}
where $B_0=B(o, 12\kappa \eta r)$ and 
$B_1=B(o, (36\kappa +6)\eta r)$. 
\end{theorem}
We note that by (\ref{inc2}), if $F\in \mathcal{W}_i^\prime$, then $12\kappa F \subset E_i$.
Hence, the Poincar\'e constants $ \Lambda( 3F, 3\kappa F)$ and 
$\Lambda(12F, 12\kappa  F)$ for $F\in  \mathcal{W}_i^\prime $
can be computed by using a local information of $E_i$.

\begin{proof}
By using (\ref{W'}), we have
\begin{eqnarray}
&&\!\!\!\!\!\!\!\!\!\!\!\!\! \!\!\!\!\!\!\!\!\!\!\!\!\! \int_{B(o,r)}|f-f_{B(o,r)}|^2 d\mu 
\leq \int_{B(o,r)}|f-f_{3B_0} |^2d\mu  \nonumber \\
& \leq &\int_{3B_0} |f-f_{3B_0} |^2 d\mu + \sum_{F\in \mathcal{W}^\prime} 
\int_{3F} |f-f_{3B_0} |^2 d\mu \nonumber\\
&\leq &\int_{3B_0} |f-f_{3B_0} |^2 d\mu +4 \sum_{F\in \mathcal{W}^\prime} 
\int_{3F}   |f-f_{3F} |^2d\mu \nonumber \\
&&+
4 \sum_{i=1}^k\sum_{F\in \mathcal{W}_i^\prime}\int_{3F} |f_{3F} -f_{3F_1} |^2d\mu +
4 \sum_{i=1}^k \sum_{F\in \mathcal{W}_i^\prime}  \int_{3F}
| f_{3F_1} - f_{3B_0} |^2   d\mu \nonumber\\
&=&I+4II+4\sum_{i=1}^k III_i+4\sum_{i=1}^k IV_i,
\label{decomposition of I}
\end{eqnarray}
where $F_1$ is an element of the string $\mathcal{F}(F)$.
We estimate terms $I, II, III_i$ and $IV_i$ in (\ref{decomposition of I}) separately.
\subsubsection*{Estimate of $I$}
By using (\ref{inc1}), we obtain
\begin{eqnarray*}
I=\int_{3B_0} |f-f_{3B_0} |^2 d\mu &\leq & \Lambda(3B_0, 3\kappa B_0) 
\int_{3\kappa B_0} |\nabla f |^2 d \mu \\
&\leq & \Lambda(3B_0, 3\kappa B_0) 
\int_{B(o,r)} |\nabla f |^2 d \mu,
\end{eqnarray*}
which gives a desired bound.
\subsubsection*{Estimate of $II$}
By using (W2) and (W4), we obtain
\begin{eqnarray*}
II = \sum_{F\in \mathcal{W}^\prime} 
\int_{3F}   |f-f_{3F} |^2 d\mu&\leq & \sum_{F \in \mathcal{W}^\prime} 
\Lambda(3F, 3\kappa F) \int_{3\kappa F} |\nabla f |^2 d\mu \\
&\leq & \max_{F \in \mathcal{W}^\prime } \Lambda(3F, 3\kappa F) 
\sum_{F \in \mathcal{W}^\prime} \int_{3\kappa F} | \nabla f |^2 d\mu \\
&\leq &N(\eta, 3\kappa) \max_{F \in \mathcal{W}^\prime }\Lambda(3F, 3\kappa F) 
 \int_{B(o,r)} | \nabla f |^2 d\mu,
\end{eqnarray*}
which gives a desired bound.
\subsubsection*{Estimate of $III_i$}
For $F \in \mathcal{W}_i^\prime$, let $F_j, F_{j+1} \in \mathcal{F}(F)$ ($j\neq 0$). 
Then the Cauchy-Schwartz inequality implies that
\begin{small}
\[
| f_{3F_{j+1}}-f_{3F_j} | \leq \frac{1}{\mu( 3F_{j+1} )^{1/2} \mu( 3F_{j} )^{1/2}} 
\left( 
\int_{3F_{j+1} \times 3F_j }\!\!\!\!\!\!\! |f(x)-f(y) |^2 d\mu (x)d\mu(y)
\right)^{1/2} .
\]
\end{small}
We note that for any open set  $D \subset M$,
\begin{equation}
\int_{D\times D}|f(x)-f(y)|^2 d\mu (x)d\mu (y)=2\mu(D) \int_{D} |f-f_D|^2 d\mu.
\label{variance}
\end{equation}
Moreover, by using (\ref{Fj}), (\ref{variance}), and the volume doubling property of $M_i$, we have
\[
| f_{3F_{j+1}}-f_{3F_j} |
\leq  \left( \frac{C}{\mu(F_j)} \int_{12 F_j } |f-f_{12 F_j} |^2d\mu \right)^{1/2}.
\]
Then we obtain
\begin{eqnarray*}
|f_{3F}-f_{3F_1}|&\leq & \sum_{j=1}^{l^\prime (F)-1}| f_{3F_{j+1}}-f_{3F_j} | \\
&\leq & \sum_{j=1}^{l^\prime (F)-1}
\left( \frac{C}{\mu(F_j)} \int_{12 F_j } |f-f_{12F_j} |^2d\mu \right)^{1/2}\\
&\leq& \sum_{j=1}^{l^\prime (F)-1}
\left( \frac{C\Lambda(12F_j, 12\kappa  F_j )}{\mu(F_j)} \int_{12\kappa  F_j } |\nabla f |^2d\mu \right)^{1/2}.
\end{eqnarray*}

By using \cite[Lemma 5.3.8]{SC LNS}, for any $F_j   \in \mathcal{F}(F)$ with $j \neq 0$, 
\[
F \subset A F_j \cap E_i,
\]
where $A=8+4\eta+\eta^{-1}$. 
Then we obtain
\begin{small}
\begin{eqnarray*}
III_i 
\!\!\!\!\!
&=& \!\!\!\!\!\!\!\! \sum_{F\in \mathcal{W}_i^\prime}\int_{3F} |f_{3F} -f_{3F_1} |^2= \int_{B\cap E_i} \sum_{F \in \mathcal{W}_i^\prime} 
\chi_{3F}(x) |f_{3F}-f_{3F_1}|^2 d\mu (x)\\
& \!\!\!\!\!\!\!\!\!\!\!\!\!\!\!\leq&\!\!\!\!\!\!\!\!\!\!\!\! C \!\! \int_{B\cap E_i} \sum_{F \in \mathcal{W}_i^\prime} 
\chi_{3F}(x) \left( \sum_{j=1}^{l^\prime (F)-1}
\!\!\!\!\Big( \frac{\Lambda(12F_j, 12\kappa  F_j )}{\mu(F_j)} 
\int_{12\kappa F_j }
\!\!\!\!\!\!\!
 |\nabla f |^2d\mu \Big)^{1/2} \right)^2 \!\!d\mu (x)\\
&\!\!\!\!\!\!\!\!\!\!\!\!\!\!\! =&\!\!\!\!\!\!\!\!\!\!\!\!C\!\! \int_{B\cap E_i} \!\sum_{F \in \mathcal{W}_i^\prime} 
\chi_{3F}(x) \Big( \sum_{j=1}^{l^\prime (F)-1}
\!\!\!\!\Big( \frac{\Lambda(12F_j, 12\kappa F_j )}{\mu(F_j)} 
\int_{12\kappa  F_j } 
\!\!\!\!\!\!\!
|\nabla f |^2d\mu \Big)^{1/2} 
\!\!\!\!\!\!
\chi_{A F_j \cap E_i}(x)
\Big)^2 \!\!d\mu (x)\\
&\!\!\!\!\!\!\!\!\!\!\!\!\!\!\! \leq &\!\!\!\!\!\!\!\!\!\!\!\!C \!\!\int_{B\cap E_i} \sum_{F \in \mathcal{W}_i^\prime} 
\chi_{3F}(x) \Big( \sum_{G \in \mathcal{W}_i^\prime }
\Big( \frac{\Lambda(12G, 12\kappa G )}{\mu(G)}
 \int_{12\kappa G }
 \!\!\!\!\!\!\!
  |\nabla f |^2d\mu \Big)^{1/2} 
 \!\!\!\!\!\!\!
\chi_{A G \cap E_i}(x) \Big)^2 \!\!d\mu (x)\\
&\!\!\!\!\!\!\!\!\!\!\!\!\!\!\! \leq&\!\!\!\!\!\!\!\!\!\!\!\! CN(\eta, 3)\max_{G \in \mathcal{W}_i^\prime} \Lambda(12G, 12\kappa  G)
\!\! \int_{B\cap E_i} \!\!\!
\Big( \!\!\sum_{G \in \mathcal{W}_i^\prime } \!\! \Big( \frac{1}{\mu(G)}
\!\!
 \int_{12\kappa  G }
 \!\!\!\!\!\!\!
  |\nabla f |^2d\mu \Big)^{1/2} 
 \!\!\!\!\!\!\!
\chi_{A G \cap E_i}(x)\Big)^2 \!\! d\mu (x)\\
&\!\!\!\!\!\!\!\!\!\!\!\!\!\!\! =&\!\!\!\!\!\!\!\!\!\!\!\! CN(\eta, 3)\max_{G \in \mathcal{W}_i^\prime} \Lambda(12G, 12\kappa G)
 \left\| \sum_{G \in \mathcal{W}_i^\prime } a_G \chi_{A G \cap E_i} \right\|_{L^2(E_i)}^2,
\end{eqnarray*}
\end{small}
where $\chi_{D}$ is the characteristic function of a set $D$ and
\[
a_G=\left( \frac{1}{\mu(G)}
 \int_{12\kappa  G } |\nabla f |^2d\mu \right)^{1/2} .
\]
Since $E_i$ satisfies (VD),  applying Lemma \cite[5.3.12]{SC LNS}, there
exists a constant $C=C(A)>0$ such that 
\[
\left\Vert \sum_{G\in \mathcal{W}_{i}^\prime  }a_{G}\chi_{A G\cap E_i}\right\Vert
_{L^{2}(E_i)}\leq C\left\Vert \sum_{G\in \mathcal{W}_{i}^\prime }a_{G} \chi_{G}\right\Vert _{L^{2}(E_i)}.
\]%
By using (W1), (W4) and (\ref{inc2}), we obtain
\begin{small}
\begin{eqnarray*}
III_i &\!\!\! \leq & \!\!\! C^\prime N(\eta, 3) \max_{F\in \mathcal{W}_i^\prime }
\Lambda (12F, 12\kappa F) 
\int_{B\cap E_{i}}\left( \sum_{F\in \mathcal{W}_{i}^\prime }a_{F}\chi_{F}(x)\right) ^{2}d\mu (x) \\
&\!\!\! =&\!\!\! C^\prime N(\eta, 3)\max_{F\in \mathcal{W}_i^\prime }\Lambda (12 F, 12\kappa F) 
\!\!
 \int_{B\cap E_{i}}\sum_{F\in \mathcal{W}_{i}^\prime }\frac{1}{\mu(F)}\int_{12\kappa F}
\!\! |\nabla f|^{2}d\mu \chi_{F}(x)d\mu (x) \\
&\!\!\! =& \!\!\! C^\prime N(\eta, 3) 
\max_{F\in \mathcal{W}_{i}^\prime }\Lambda (12 F,12\kappa  F)\sum_{F\in 
\mathcal{W}_i^\prime }\int_{12\kappa F}|\nabla f|^{2}d\mu \\
&\!\!\! \leq &\!\!\! C^\prime N(\eta, 3)N(\eta, 12\kappa )
 \max_{F\in \mathcal{W}_{i}^\prime }\Lambda (12 F,12\kappa F)
 \int_{B}|\nabla f|^{2}d\mu ,
\end{eqnarray*}%
\end{small}
which gives a desired bound.
\subsubsection*{Estimate of $IV_i$}
Because $3F_1,  3B_0 \subset B_1$,  
by using the same argument as in the estimate of $III_i$, we obtain
\begin{eqnarray*}
IV_i &=&\sum_{F\in \mathcal{W}_i^\prime}  \int_{3F}
| f_{3F_1} - f_{3B_0} |^2   d\mu \\
&\leq & 
\sum_{F \in \mathcal{W}_i^\prime} \frac{\mu(3F)}{\mu(3F_1) \mu(3B_0)} \int_{B_1 \times B_1}
|f(x)-f(y) |^2 d\mu (x) d\mu (y)\\
&=&2\left( \sum_{F \in \mathcal{W}_i^\prime} \frac{\mu(3F)}{\mu(3F_1)} \right)
\frac{\mu(B_1)}{\mu(3B_0)} \int_{B_1 }
|f-f_{B_1} |^2 d\mu .
\end{eqnarray*}
Here we recall that the central ball satisfies (VD). Indeed, 
\begin{eqnarray}
\mu(B(o,s)) &\simeq &\mu(B_1(o_1, s)) +\mu(B_2(o_2, s) ) +\cdots \mu(B_k(o_k, s)) \nonumber \\
&\simeq &\max_{i} \mu(B_i(o_i, s)),  
\label{center VD}
\end{eqnarray}
where $B_i(o_i, s)$ is the geodesic ball in $M_i$ centered at $o_i \in K_i$. 
Since each $B_i(o_i,s)$ satisfies (VD), so does $B(o, s)$ by (\ref{center VD}). 
Hence $\mu(B_1)/\mu(3B_0)$ is bounded in $r$.
Now we estimate $\mu(3F_1)$ from below. 
By using (W3), Lemma \ref{lemF1} and (\ref{inc1}), 
for any $F\in \mathcal{W}_i^\prime$
\begin{eqnarray*}
r(F_1) &=&\eta d(F_1, B(o,r)^c)  \\
&\geq &\eta d(B_1, B(o,r)^c)=\eta  (1-(36\kappa +6)\eta)r,
\end{eqnarray*}
which implies that 
\[
A^\prime r(F_1) \geq 2r,
\]
where $A^\prime =\frac{2}{\eta(1-(36\kappa +6)\eta)} $. Then we obtain
\[
A^\prime F_1\supset  B(o,r).
\]
Since each end $E_i$ is doubling,  we obtain
\[
\mu(3F_1) \geq C^{-1} \mu(A^\prime F_1 \cap E_i ) \geq C^{-1} \mu(B(o,r)\cap E_i).
\]
By using (W2), we obtain for all $i=1, \ldots, k$
\[
\sum_{F \in \mathcal{W}_i^\prime} \frac{\mu(3F)}{\mu(3F_1)} 
\leq \frac{C}{\mu( B(o,r) \cap E_i)} \sum_{F \in \mathcal{W}_i^\prime} \mu(3F) \leq  C.
\]
Combining the above estimates, we obtain
\begin{eqnarray*}
IV_i &\leq & C^\prime \int_{B_1} |f-f_{B_1}|^2 d\mu \\
&\leq &C^\prime \Lambda(B_1, \kappa B_1) \int_{\kappa B_1} |\nabla f|^2 d\mu.
\end{eqnarray*}
Since $\kappa B_1 \subset B(o,r)$ by (\ref{inc1}), we conclude a desired estimate of $IV_i$.
\end{proof}

We use the following corollary to show the upper bound of the Poincar\'e constant  
(i.e., lower bound of the spectral gap) of central balls in Theorems \ref{main non-parab} 
and \ref{main parab}.
\begin{corollary}
Let $M$ be a manifold with nice ends $M_1, \ldots , M_k$, where each parabolic end satisfies (RCA) . If there is at least one parabolic end, 
assume that $M$ satisfies {\rm{(COE)}}. Then for  any large $r>0$
\[
\Lambda(B(o,r)) \lesssim    V_n(r) h_n (r),
\]
that is, 
\[
\lambda(B(o,r)) \gtrsim   \frac{1}{ V_n(r) h_n (r)},
\]
where $n$ is the index of the second largest end defined in (\ref{second largest end}).
\end{corollary}
\begin{proof}
Because $3 \kappa F, 12\kappa F \subset E_i$ for all $F \in  \mathcal{W}_i^\prime$
by (\ref{inc2}), 
the assumption of (VD) and  (PI) on each $M_i$ implies that  for all
$F \in \mathcal{W}^\prime$
\[
\Lambda (3F, 3\kappa F), \Lambda(12 F, 12\kappa F) \lesssim  r^2.
\]
Using the estimates in section  \ref{SecUBPoincare}, we obtain
\[
 \Lambda (3B_0 , 3\kappa B_0)+\Lambda(B_1, \kappa B_1) \lesssim
  V_n (r)h_n(r).
\]
Since the inequality $V_i (r) h_i(r) \gtrsim r^2$ is always true by (VD), we
conclude the corollary from Theorem \ref{main Whitney} immediately.
\end{proof}

For the estimates of Poincar\'e constants of central balls 
on some examples of manifolds with ends, 
we refer to Section \ref{SecExamples}.

\section*{Acknowledgments}
The first and second authors would like to express their gratitude for the
hospitality and support by the Institute of Mathematical Sciences of the
Chinese University of Hong Kong and School of Mathematical Sciences, Nankai
University. The authors would like to thank the anonymous referee for their valuable comments and suggestions.





\end{document}

%% file: macro.tex
\usepackage{hyperref}

\def\subjclass#1{\par{\footnotesize {\bf 2020 Mathematics Subject Classification:} #1}}
\def\keywords#1{\par{\footnotesize {\bf Keywords:} #1}}

\def\FRAME#1#2#3#4#5#6#7#8
{
 \begin{figure}[ptbh]
 \begin{center}
 \includegraphics[height=#3]{#7}
 \caption{#5}
 \label{#6}
 \end{center}
 \end{figure}
}